\documentclass[reqno]{amsart}

\usepackage{amsmath, amssymb, amsthm, mathtools, enumitem}
\usepackage{accents, xcolor}
\usepackage{cancel}
\usepackage{amssymb}
\usepackage{amsthm}
\usepackage{amsfonts}
\usepackage{amsmath}
\usepackage{pmboxdraw}
\usepackage{verbatim} % for comment
\usepackage{graphicx}
\usepackage{color}
\usepackage[colorlinks=true, citecolor=blue, filecolor=black, linkcolor=black, urlcolor=black]{hyperref}
\usepackage{cite}
\usepackage[normalem]{ulem}
\usepackage{subcaption}
\usepackage{todonotes}
\usepackage{kantlipsum}
\allowdisplaybreaks
\usepackage{bbm}
\usepackage{tikz}

\newtheorem{theorem}{Theorem}[section]
\newtheorem{lemma}[theorem]{Lemma}
\newtheorem{proposition}[theorem]{Proposition}

\numberwithin{equation}{section}

\theoremstyle{remark}
\newtheorem{remark}{Remark}

\newcommand{\erfc}{\operatorname{erfc}}

\newcommand{\re}{\operatorname{Re}}

\newcommand{\im}{\operatorname{Im}}

\newcommand{\essinf}{\operatorname*{ess\,inf}}

\def\C{\mathbb{C}}
\def\D{\mathbb{D}}
\def\E{\mathbb{E}}
\def\HH{\mathbb{H}}
\def\P{\mathbb{P}}
\def\R{\mathbb{R}}

\def\rmc{\mathrm{c}}
\def\rmr{\mathrm{r}}

\def\tS{\mathbb{H}}

\begin{document}

\title[Upper Tail Large Deviations of the elliptic Ginibre ensembles]{Upper Tail Large Deviations for Extremal Eigenvalues of \\ the Real, Complex and Symplectic Elliptic Ginibre Matrices}

\author{Sung-Soo Byun}
 \address{Department of Mathematical Sciences and Research Institute of Mathematics, Seoul National University, Seoul 151-747, Republic of Korea}
 \email{sungsoobyun@snu.ac.kr}

\author{Yong-Woo Lee} 
\address{Department of Mathematical Sciences, Seoul National University, Seoul 151-747, Republic of Korea}
\email{hellowoo@snu.ac.kr} 

\author{Seungjoon Oh}
 \address{Department of Mathematical Sciences, Seoul National University, Seoul 151-747, Republic of Korea}
 \email{seungjoonoh@snu.ac.kr}

\begin{abstract}
We consider the elliptic Ginibre ensembles in the real, complex and symplectic symmetry classes. As the matrix size tends to infinity, we derive the asymptotic behaviour of the upper tail large deviation probabilities for both the spectral radius and the rightmost eigenvalue. More generally, we obtain asymptotic formulas for the probability that an eigenvalue is found in a prescribed region outside the support of the elliptic law, thereby providing a unified framework in which the results for the spectral radius and the rightmost eigenvalue appear as special cases. The key ingredient of our analysis is the precise asymptotic behaviour of the associated one-point functions, which is of independent interest. 
\end{abstract} 

\maketitle

\section{Introduction}

Exact analytic results have played a central role in the development of random matrix theory. For the classical ensembles, particularly those with Gaussian entries, the presence of rich algebraic and integrable structures permits explicit finite-dimensional formulas for a broad class of observables, including correlation functions and various deviation probabilities; see \cite{Fo10} for an account of this aspect of the literature up until around 2010. Over the past decades, these structures have given rise to striking exact identities, from determinantal and Pfaffian formulas to characterisations in terms of Painlevé equations. Remarkably, new exact formulas and structural features continue to appear, revealing additional layers of integrability even in models long regarded as classical; see e.g. \cite{Fo25,Fo25b} and references therein. 

A notable example in which such exact identities play a decisive role is the study of extremal eigenvalue statistics. 
In the classical Gaussian invariant ensembles, namely the Gaussian orthogonal, unitary, and symplectic ensembles (GOE, GUE, and GSE) \cite[Chapter 1]{Fo10}, the  joint probability density function of the eigenvalues $\{x_j\}_{j=1}^N$ on the real line is proportional to 
\begin{equation} \label{def of jpdf of GbetaE}
\prod_{1 \le j < k \le N} |x_j-x_k|^\beta \prod_{j=1}^N e^{ -\frac{\beta N}{4} x_j^2 } \,dx_j,
\end{equation}
where the Dyson index takes the values $\beta=1,2,4$ corresponding to GOE, GUE, and GSE, respectively. It is by now classical that the empirical spectral measure associated with \eqref{def of jpdf of GbetaE} converges to Wigner's semicircle law 
\begin{equation} \label{def of semicircle}
\frac{\sqrt{4-x^2}}{2\pi} \mathbbm{1}_{[-2,2]}(x)\,dx, 
\end{equation} 
as $N \to \infty.$ At the macroscopic scale, it is well known that the largest and smallest eigenvalues converge almost surely to the right and left endpoints of the semicircle law, respectively. 
At the level of fluctuations, the typical edge behaviour is governed by the Tracy-Widom distributions, which have become a cornerstone of modern probability theory and of universality phenomena in random matrix theory. 

A natural next step is to investigate atypical events, the large deviation behaviour of the extremal eigenvalues. 
In this context, a fundamental question concerns the large-$N$ asymptotics of the probability 
\begin{equation} \label{def of extremal prob GbetaE}
\P\Big[ \max_{j=1, \ldots, N} x_j \geq s \Big] \qquad \textup{for } s > 2, 
\end{equation}
where the restriction $s>2$ reflects the fact that $2$ is the right edge of the semicircle law~\eqref{def of semicircle}. 
As expected for the upper tail event, this probability decays exponentially fast as $N \to \infty$. 
It was shown in~\cite{MV09} (see also \cite[Theorem 6.2]{BDG01} for the case $\beta=1$) that for $s>2$,
\begin{equation} \label{result in MV09}
\lim_{ N \to \infty } \frac{1}{\beta N} \log \P\Big[ \max_{j=1, \ldots, N} x_j \geq s \Big]= -\Phi_1(s), 
\end{equation}
where the rate function is given by\footnote{We remark that in the original paper \cite[Eq.~(13)]{MV09}, the function $\Phi_1(s)$ is expressed as 
$$
\Phi_1(s)= \frac{s^2-2}{4} - \log s + \frac{1}{2s^2} {}_3F_2\Big( \genfrac{}{}{0pt}{}{1,1,3/2}{2,3} \Big| \frac{4}{s^2} \Big).
$$
This representation can be simplified to the form given in \eqref{def of rate Phi_1}. 
} 
\begin{equation} \label{def of rate Phi_1}
\Phi_1(s) := \frac{s\sqrt{s^2 - 4}}{4} - \log \bigg( \frac{s+ \sqrt{s^2 -4}}{2} \bigg);
\end{equation}
see also \eqref{int rep of Phi 1} below for an integral representation of $\Phi_1$.
Here, the subscript $1$ is introduced for later comparison. We remark that in \eqref{result in MV09}, the left-hand side is scaled by $\beta N$, indicating that the probability of finding an eigenvalue far from its typical location decreases as $\beta$ increases. This is consistent with the interpretation of the Dyson index $\beta$ as an inverse temperature: as $\beta$ increases, the eigenvalues exhibit stronger rigidity, tending to remain closer to their expected positions \cite{BEY14}.  

Large deviation probabilities of this type have been extensively studied in the literature.
A list of works on classical ensembles includes \cite{DM06,DM08,BDG01,AGKWW14,KC10,PS16,VMB07,WG13,WG14,WKG15,MV09,Ku19,Fo12,FW12,DR16,MS14,CFLV18,BEMN11}.
We also refer to the more recent developments in 
\cite{BCMS25,BSY25,CG21,DS17}, where precise large-$N$ asymptotic 
expansions of large deviation probabilities have been established. 
See also \cite{BFMS26,BF25a} and  references therein for a general large deviation framework formulated from the perspective of equilibrium statistics. 
We emphasise that such probabilities play an important role 
not only within random matrix theory, but also in integrable probability models, owing to several exact correspondences between these frameworks; see e.g. \cite{Joh00,BCMS25,BR01,FW04} and references therein.

\medskip 
 
Beyond the Hermitian setting, the fundamental models in non-Hermitian random matrix theory are the Ginibre ensembles in the real, complex, and symplectic symmetry classes, denoted by GinOE, GinUE, and GinSE, respectively. 
These ensembles consist of random matrices with independent and identically distributed Gaussian entries, real in the GinOE case, complex in the GinUE case, and quaternionic in the GinSE case; see \cite{BF25} for a recent review.  
Unlike the Hermitian case, where the joint probability densities of the GOE, GUE, and GSE admit the unified $\beta$-ensemble form \eqref{def of jpdf of GbetaE}, no such unified structure exists for the Ginibre ensembles.
For instance, the joint probability density of the complex eigenvalues $\{z_j\}_{j=1}^N$ of the GinUE is proportional to 
\begin{equation} \label{def of JPDF GinUE}
\prod_{ 1 \le j < k \le N } |z_j-z_k|^2 \prod_{j=1}^N e^{ -N |z_j|^2 } d^2z_j,
\end{equation}
whereas its GinSE counterpart is proportional to 
\begin{equation} \label{def of JPDF GinSE}
\prod_{ 1 \le j < k \le N } |z_j-z_k|^2 |z_j-\bar{z}_k|^2 \prod_{j=1}^N  |z_j-\bar{z}_j|^2 e^{ -2N |z_j|^2 } d^2z_j. 
\end{equation}
The situation for the GinOE is more involved.
In this case, the eigenvalue distribution is no longer absolutely continuous 
with respect to the two-dimensional Lebesgue measure, due to the non-trivial probability of having purely real eigenvalues; see Figure~\ref{Fig_samplings} (A).  Nevertheless, upon conditioning on the number of real and complex eigenvalues, the corresponding joint density admits an explicit representation; see \cite[Section 7.1]{BF25}.

Despite the distinct integrable structures appearing 
in their joint eigenvalue densities, all three Ginibre ensembles share the same limiting empirical distribution, namely the circular law
\begin{equation} \label{def of circular law}
\frac{1}{\pi}\mathbbm{1}_{ \mathbb{D} }(z) \, d^2z ,\qquad \mathbb{D}=\{ z\in \C :|z| \le 1 \}.   
\end{equation}
As before, it is natural to investigate the statistics of extremal eigenvalues in this non-Hermitian setting.
Among the most fundamental observables are the \textit{spectral radius} and the \textit{rightmost eigenvalue} given by 
\begin{equation}
\max_{j=1, \ldots, N} |z_j|, \qquad \max_{j=1, \ldots, N} \re z_j,
\end{equation}
respectively. 
By definition, the spectral radius quantifies radial deviations from the circular law and determines whether eigenvalues exit the limiting support.
In contrast, the rightmost eigenvalue plays a crucial role in continuous-time dynamical systems, as stability is determined by the sign of its real part.
A paradigmatic example is May's work \cite{May72} on ecosystem stability, where large ecological networks are modelled by random interaction matrices.
In this setting, stability of the equilibrium requires that all eigenvalues lie in the left half-plane.
Therefore, estimating the probability of atypical fluctuations that push eigenvalues beyond this stability boundary is of fundamental importance, both in random matrix theory and in applications to complex systems; see also \cite{BFK21,FK16}.

For the Ginibre ensembles, as a non-Hermitian counterpart of 
\eqref{def of extremal prob GbetaE}, one asks for the asymptotic behaviour of the probabilities
\begin{equation} \label{def of extremal prob Ginibre}
 \P\Big[ \max_{j=1, \ldots, N} |z_j| \geq s \Big], \qquad  \P\Big[ \max_{j=1, \ldots, N} \re z_j \geq s \Big] \qquad \textup{for } s > 1, 
\end{equation} 
where the restriction $s>1$ now reflects the support of the circular law 
\eqref{def of circular law}. 
For the GinUE, the asymptotic behaviour of the spectral radius was obtained in \cite{CMV16,LGMS18}, showing that
\begin{equation} \label{result in CMV16 and LGMS18}
\lim_{N \to \infty} \frac{1}{\beta N}\log \P\Big[ \max_{j=1, \ldots, N} |z_j| \geq s  \Big] =  - \Phi_0(s) ,
\end{equation}
where $\beta=2$ and 
\begin{equation}  \label{def of rate Phi_0}
\Phi_0(s):= \frac{1}{2} (s^2 - 2 \log s - 1). 
\end{equation}
As before, we introduce the subscript $0$ for later convenience. 
A key advantage of considering the GinUE and the spectral radius is that the
problem preserves radial symmetry. As a consequence, the determinantal
structure allows one to express this probability in terms of a single
summation involving the incomplete Gamma function, for which a detailed
asymptotic analysis can be carried out; see
\cite{Fo92,BP26,Ch22,Ch23,Ch23a,AZ15,AFLS25,APS09,AS13,ACM24,ACCL24} and references therein for a line of research in this direction. 

In contrast, for the GinOE and GinSE, as well as for the rightmost eigenvalue, the radial symmetry is broken so that one needs an alternative approach. 
For the GinOE and GinUE, it was shown in the recent work of Xu and Zeng 
\cite{XZ24} that
\begin{align} \label{result in XZ24}
\lim_{N \to \infty} \frac{1}{\beta N}\log \P\Big[ \max_{j=1, \ldots, N} |z_j| \geq s  \Big] = \lim_{N \to \infty} \frac{1}{\beta N}\log \P\Big[ \max_{j=1, \ldots, N} \re z_j \geq s  \Big] = - \Phi_0(s),
\end{align}
where $\beta=1$ and $\beta=2$ correspond to the GinOE and GinUE, respectively.
We refer to \cite{CESX22,CESX23,Be10,ACC25,CP14,CFLV17,BS23,Seo20,BL24a} and references therein for further related results, including fluctuations described by the Gumbel distribution and lower tail large deviation probabilities for the spectral radius.

In contrast to the extensive body of work on large deviation probabilities 
for Hermitian random matrix ensembles discussed above, 
the corresponding theory in the non-Hermitian setting remains comparatively underdeveloped, despite the growing interest in non-Hermitian random matrix theory in recent years. 
In particular, the study of upper tail large deviation probabilities for general non-Hermitian ensembles extending the Ginibre models still remains largely open. 
Indeed, even for the Ginibre ensembles themselves, the corresponding result is still unavailable in the symplectic case. 

In this work, we contribute to this direction by employing a unified framework of exact finite-$N$ analysis to derive upper tail large deviation probabilities for the elliptic Ginibre ensembles in all three  symmetry classes.
Before presenting our results in detail in the next section, we summarise our main findings as follows.

\begin{itemize}
\item For the elliptic Ginibre matrices in all three symmetry classes, we derive in Theorem~\ref{Thm. LDP maxEV} the upper tail large deviation probabilities for both the spectral radius and the rightmost eigenvalue. 
Our results interpolate between the Hermitian and the non-Hermitian results, \eqref{result in MV09} and \eqref{result in XZ24}.
\smallskip
\item In Theorem~\ref{Thm. LDP maxEV_general domain}, which extends Theorem~\ref{Thm. LDP maxEV}, we establish the large deviation probabilities for the event that an eigenvalue is located in a prescribed domain outside the typical limiting spectrum. 
\smallskip
\item As an intermediate step towards the proof of Theorem~\ref{Thm. LDP maxEV_general domain}, in Propositions~\ref{Prop_eGinUE asymp}, 
\ref{Prop_eGinOE asymp}, and \ref{Prop_eGinSE asymp}, we establish the precise asymptotic behaviour of the one-point functions of the elliptic Ginibre ensembles outside the limiting spectrum, which are of independent interest.
\end{itemize}

\subsection*{Acknowledgements} This work was supported by the National Research Foundation of Korea grants (RS-2023-00301976, RS-2025-00516909). This work was motivated by \cite{XZ24} and we would like to thank the authors, Yuanyuan Xu and Qiang Zeng, for stimulating discussions. We also thank Peter J. Forrester, Satya N. Majumdar, and Gregory Schehr for their interest and for helpful comments on an earlier version of the manuscript.

\section{Main results}

We begin by introducing our model, the elliptic Ginibre ensembles. Let $\mathcal{G}$ denote the Ginibre ensemble, the $N \times N$ random matrix with independent centred Gaussian entries of variance $1/N$. We consider the real, complex, and quaternionic cases, corresponding to the GinOE, GinUE, and GinSE, respectively. In the quaternionic case, we adopt the standard $2 \times 2$ complex matrix representation; see \cite{BF25} for further details. 

By construction, the elliptic Ginibre ensembles interpolate between the 
Ginibre ensembles and their Hermitian counterparts. 
They are defined as a convex linear combination of the Hermitian and 
anti-Hermitian parts of a Ginibre matrix: for a given non-Hermiticity parameter $\tau \in [0,1)$,
\begin{equation}
    \mathcal{X}_\tau :=  \frac{ \sqrt{1+\tau} }{2} (\mathcal{G}+\mathcal{G}^\dagger) + \frac{ \sqrt{1-\tau} }{2} (\mathcal{G}-\mathcal{G}^\dagger) .
\end{equation}
Depending on the underlying symmetry class of the Ginibre matrix, 
we denote the corresponding ensembles by eGinOE, eGinUE, and eGinSE. 
We refer to \cite[Sections 2.3, 7.9 and 10.5]{BF25} for comprehensive reviews on these ensembles. Note that, by definition, the extremal cases $\tau = 0$ and $\tau \to 1$ recover the Ginibre ensembles and the corresponding Gaussian invariant ensembles, respectively. 

The joint probability density functions of the elliptic Ginibre ensembles take forms analogous to that of their Ginibre counterparts. For instance, in the cases of the eGinUE and eGinSE, the joint densities are given by \eqref{def of JPDF GinUE} and \eqref{def of JPDF GinSE}, respectively, with the Gaussian term $|z|^2$ replaced by the corresponding external potential: 
\begin{equation} \label{def of V eGinibre}
 V(z) := \frac{1}{1-\tau^2}\Big( |z|^2-\tau \re z^2\Big).
\end{equation}  
Note that when $\tau=0$, the potential $V$ reduces to the quadratic potential $|z|^2$ corresponding to the Ginibre ensembles. On the other hand, as $\tau \to 1$, it formally recovers the Gaussian potential in \eqref{def of jpdf of GbetaE}, since
\begin{equation}
\lim_{ \tau \to 1 } V(x+iy) = \begin{cases}
x^2 &\textup{if } y =0,
\smallskip 
\\ 
+\infty & \textup{otherwise}.
\end{cases}
\end{equation}
Here, the value $+\infty$ reflects the fact that, in this limit, the probability of observing eigenvalues off the real axis vanishes, which is consistent with the matrix model becoming Hermitian. 

Turning to the averaged macroscopic behaviour of the eigenvalues $\{z_j\}_{j=1}^N$, as $N \to \infty$, the empirical spectral distribution converges weakly to
\begin{equation} \label{def of elliptic law}
\frac{1}{\pi(1-\tau^2)} \, \mathbbm{1}_{ E }(z) \, d^2z ,\qquad E:=  \Big\{ (x,y) \in \R^2: \Big( \frac{x}{1+\tau}\Big)^2 + \Big( \frac{y}{1-\tau} \Big)^2 \le 1 \Big\},
\end{equation}
which is known as the elliptic law; see Figure~\ref{Fig_samplings}.

\begin{figure}[t]
\begin{subfigure}{0.32\textwidth}
        \begin{center}
            \includegraphics[width=\linewidth]{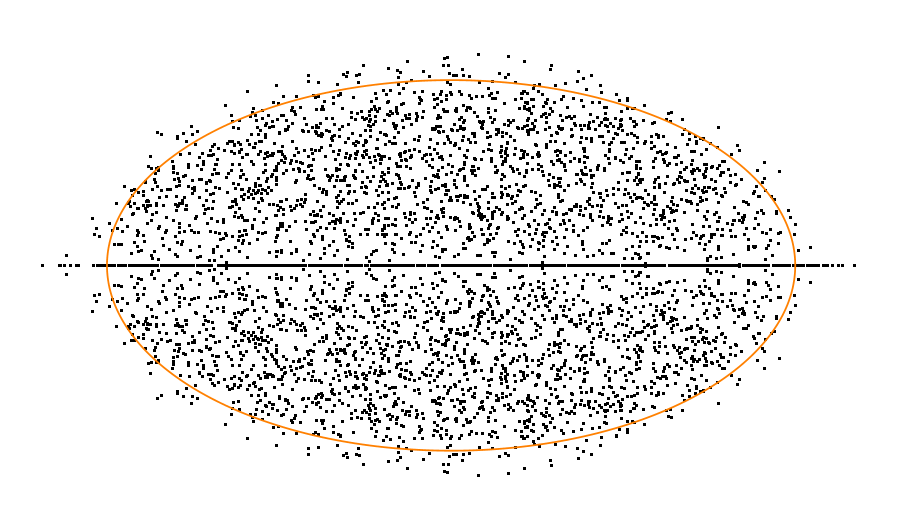}
        \end{center}
        \subcaption{ The eGinOE }
    \end{subfigure}
    \begin{subfigure}{0.32\textwidth}
        \begin{center}
            \includegraphics[width=\linewidth]{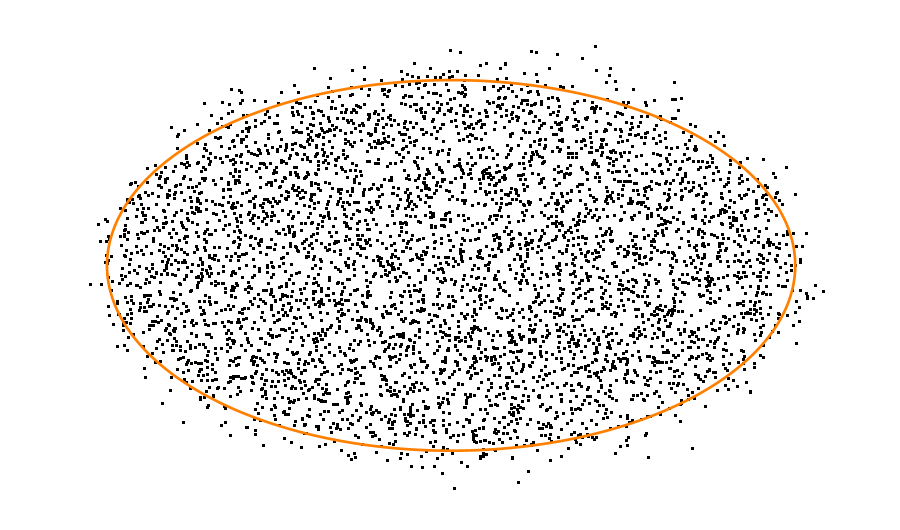}
        \end{center}
        \subcaption{ The eGinUE  }
    \end{subfigure}
    \begin{subfigure}{0.32\textwidth}
        \begin{center}
            \includegraphics[width=\linewidth]{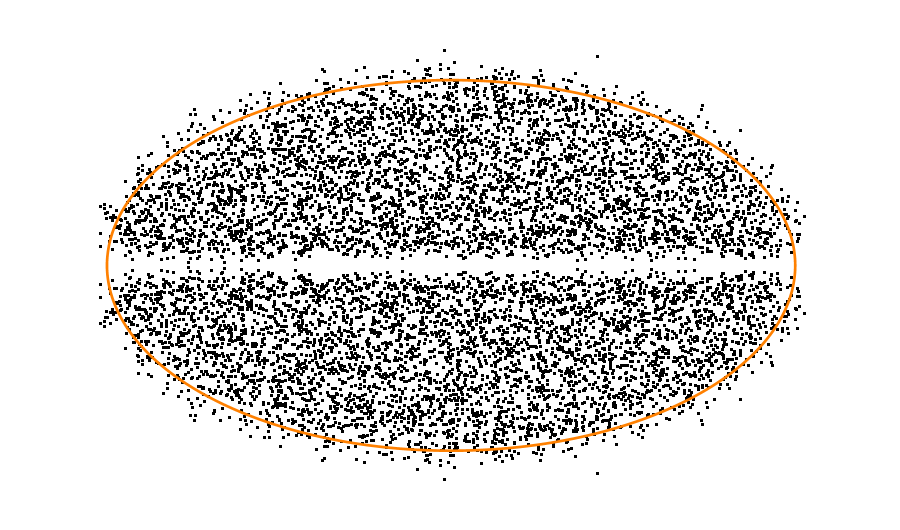}
        \end{center}
        \subcaption{ The eGinSE  }
    \end{subfigure}
    \caption{ Spectra of 50 samples from the elliptic Ginibre ensemble with $N=100$ and  $\tau = 0.3$. The solid orange lines indicate the boundary of the ellipse  \eqref{def of elliptic law}. } 
    \label{Fig_samplings}
\end{figure}

As mentioned earlier, the eigenvalue statistics of the eGinOE exhibit a 
distinctive feature compared to the other symmetry classes: with 
non-trivial probability, the matrix possesses purely real eigenvalues; cf. Figure~\ref{Fig_samplings} (A).  
The study of real eigenvalues in this setting has become an active area 
of research, motivated by various applications; 
see e.g. \cite{BKLL23,BFK21,Fo24,BMS25,Kiv25,CFW25,BJLS25} for recent developments. We also refer to \cite{BN25,FIK20,FK18,FS23,Fo25a,AB23,LMS22} for related works on variants of these models. 
In fact, in the eGinOE case, the number of real eigenvalues is typically 
of order $O(\sqrt{N})$ when $\tau$ is fixed. Consequently, in the 
large-$N$ limit, the empirical spectral measure is predominantly 
determined by the complex eigenvalues and therefore converges to the 
elliptic law \eqref{def of elliptic law}.

\subsection{Spectral radius and rightmost eigenvalue}

As in the Ginibre ensemble, we investigate the asymptotic behaviour of 
probabilities of the form \eqref{def of extremal prob Ginibre}. 
In the present setting, however, the elliptic law 
\eqref{def of elliptic law} implies that the relevant regime 
corresponds to $s > 1+\tau$. More precisely, we consider 
\begin{equation} \label{def of extremal prob eGinibre}
 \P\Big[ \max_{j=1, \ldots, N} |z_j| \geq s \Big], \qquad  \P\Big[ \max_{j=1, \ldots, N} \re z_j \geq s \Big] \qquad \textup{for } s > 1+\tau. 
\end{equation} 
See Figure~\ref{Fig_LDP illustration} for an illustration of the 
situation. Our main result below establishes the asymptotic 
behaviour of these probabilities. 
For convenience, we adopt the notation
\begin{equation} \label{def of beta for eGinibres}
\beta = \begin{cases}
1 &\textup{for the eGinOE},
\smallskip 
\\
2 &\textup{for the eGinUE},
\smallskip 
\\
4 &\textup{for the eGinSE}.
\end{cases}
\end{equation}
Then we have the following.

\begin{figure}[t]
\begin{subfigure}{0.4\textwidth}
        \begin{center}
            \includegraphics[width=\linewidth]{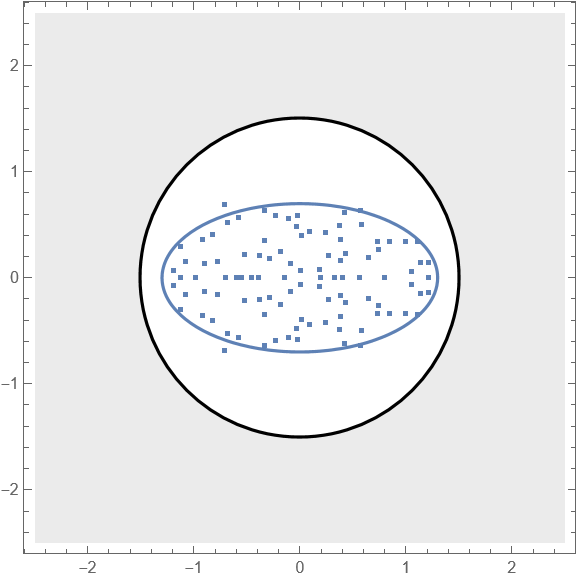}
        \end{center}
        \subcaption{ $\P[ \max |z_j| \geq s ]$ }
    \end{subfigure}
    \qquad
    \begin{subfigure}{0.4\textwidth}
        \begin{center}
            \includegraphics[width=\linewidth]{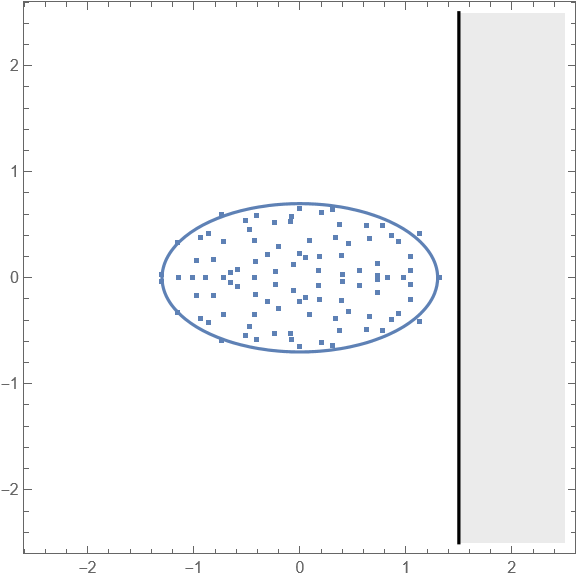}
        \end{center}
        \subcaption{ $\P[ \max \re z_j \geq s ]$  }
    \end{subfigure}
    \caption{ The plot displays the eigenvalues of the eGinOE for $N = 100$ and $\tau = 0.3$. The probability defined in \eqref{def of extremal prob eGinibre} corresponds to the event that at least one eigenvalue lies in the shaded region, where $s = 1.5$. By \eqref{def of rate Phi_tau}, for these parameter values we have $\Phi_\tau(s) \approx 0.038$. Hence, by Theorem~\ref{Thm. LDP maxEV}, the probability of this event is exponentially small, behaving as $e^{-N \Phi_\tau(s)} \approx 0.022$. }
    \label{Fig_LDP illustration}
\end{figure}

\begin{theorem}[\textbf{Upper tail large deviation probabilities for spectral radius and rightmost eigenvalue}] \label{Thm. LDP maxEV}
Let $\tau \in [0,1)$ be fixed, and let $N \in \mathbb{N}$. In the eGinOE case, we further assume that $N$ is even. Suppose that $s>1+\tau$. Then we have 
    \begin{equation}
    \lim_{N \to \infty} \frac{1}{\beta N}\log \P\Big[ \max_{j=1, \ldots, N} |z_j| \geq s  \Big] = \lim_{N \to \infty} \frac{1}{\beta N}\log \P\Big[ \max_{j=1, \ldots, N} \re z_j \geq s  \Big]=  -  \Phi_\tau(s),
    \end{equation}
    where
     \begin{equation} \label{def of rate Phi_tau}
        \Phi_\tau(s) :=  \frac12 \bigg( \frac{1}{1+\tau}-\frac{1}{2\tau}\bigg) s^2+  \frac{  s \sqrt{s^2-4\tau} }{4\tau }   -  \log\bigg( \frac{s + \sqrt{s^2 -4\tau}}{2}\bigg) . 
    \end{equation}
    Here, $\beta$ is given by \eqref{def of beta for eGinibres}. 
\end{theorem}

One can observe that the rate function \eqref{def of rate Phi_tau}  reduces to \eqref{def of rate Phi_1} and \eqref{def of rate Phi_0}  when $\tau = 0$ and $\tau = 1$, respectively. 
Consequently, for $\beta = 1,2$, Theorem~\ref{Thm. LDP maxEV} interpolates between the results of \cite{MV09} ($\tau = 1$) and \cite{XZ24} ($\tau = 0$). 
On the other hand, we note that for $\beta = 4$, our result is new even in the GinSE case ($\tau = 0$). 

It is straightforward to verify that $\Phi_\tau(s) \ge 0$ for $s \ge 1+\tau$, with equality at $s = 1+\tau$. 
Moreover, the function $s \mapsto \Phi_\tau(s)$ is increasing on this domain. See Figure~\ref{Fig_Phi tau} for the graphs of $s \mapsto \Phi_\tau(s)$. 

In addition, note that for $\tau \in [0,1)$,
\begin{equation}
\Phi_\tau(s) = \frac{ (s-1-\tau)^2 }{1-\tau^2} +O\big( (s-1-\tau)^3 \big), \qquad \textup{as } s \to 1+\tau.    
\end{equation}
This behaviour is in contrast with the case $\tau=1$, where
\begin{equation}
\Phi_1(s) = \frac{2}{3} (s-2)^{ \frac32 } +O\big( (s-2)^{ \frac53 } \big), \qquad \textup{as } s \to 2.    
\end{equation}
The difference originates from the analytic structure of $\Phi_\tau$. 
For $\tau<1$, the point $s=1+\tau$ lies away from the branch point 
$s=2\sqrt{\tau}$ of the square-root term in \eqref{def of rate Phi_tau}, and hence $\Phi_\tau$ is analytic in a neighbourhood of $1+\tau$, leading to the 
quadratic expansion above. In contrast, when $\tau=1$ the branch point 
coincides with $s=2$, and the analyticity breaks down, resulting in the 
non-analytic $(s-2)^{3/2}$ behaviour. Such different local behaviours 
are natural and reflect the distinct tail asymptotics of the Gumbel and 
Tracy--Widom distributions, respectively. 

\begin{figure}[t]
\begin{subfigure}{0.3\textwidth}
        \begin{center}
            \includegraphics[width=\linewidth]{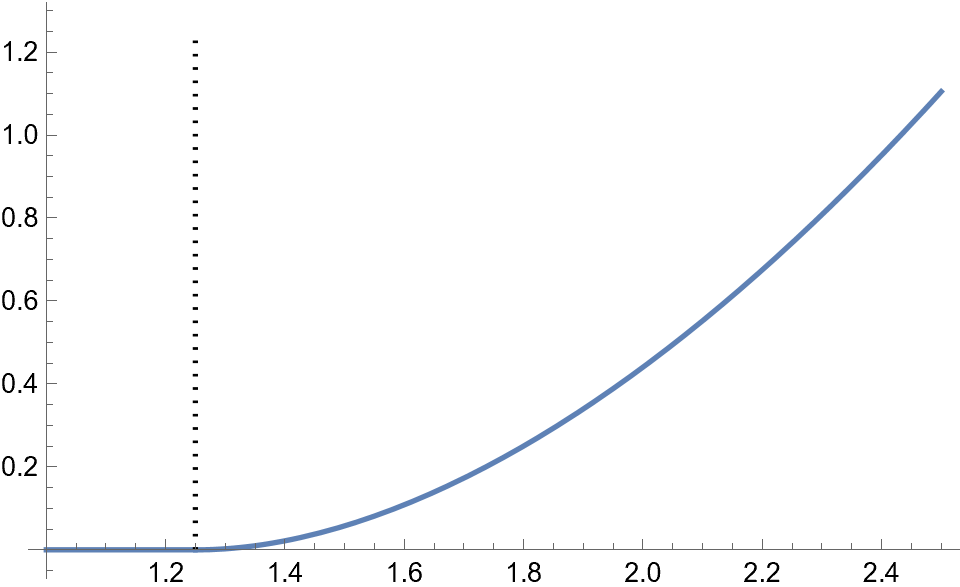}
        \end{center}
        \subcaption{ $\tau=1/4$ }
    \end{subfigure}
    \qquad
    \begin{subfigure}{0.3\textwidth}
        \begin{center}
            \includegraphics[width=\linewidth]{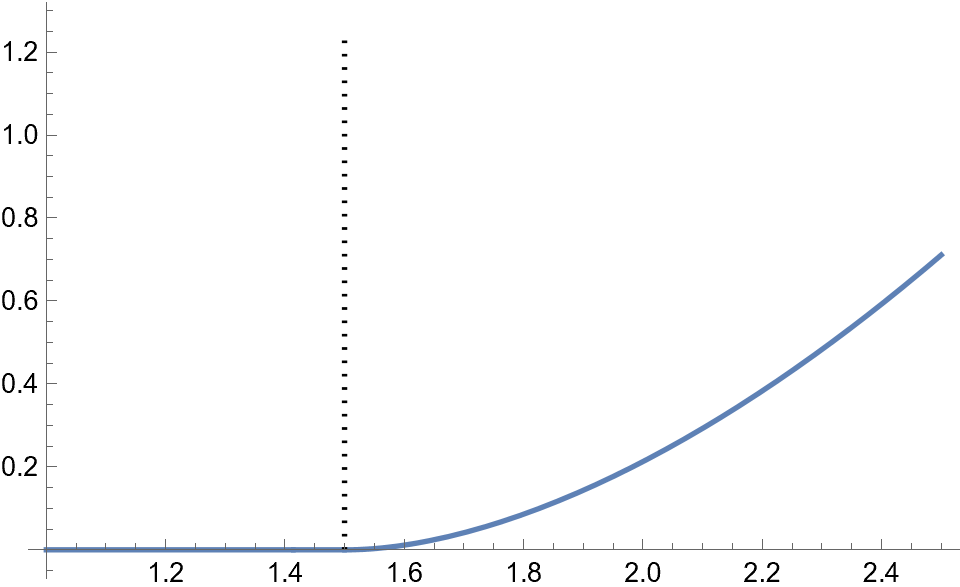}
        \end{center}
        \subcaption{  $\tau=1/2$  }
    \end{subfigure}
     \qquad
    \begin{subfigure}{0.3\textwidth}
        \begin{center}
            \includegraphics[width=\linewidth]{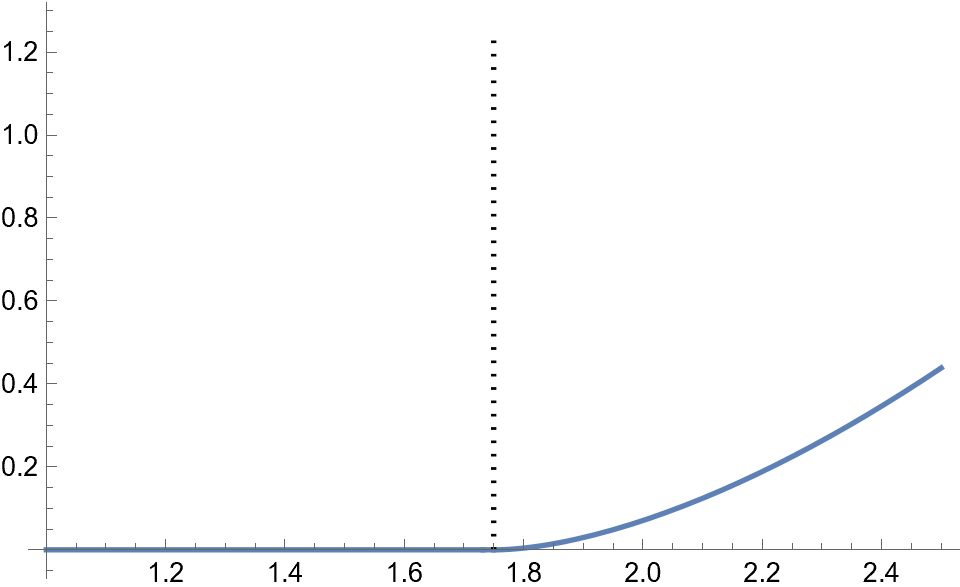}
        \end{center}
        \subcaption{ $ \tau=3/4 $  }
    \end{subfigure}
    \caption{ The plots show the graphs of $s \mapsto \Phi_\tau(s)$ for $s > 1+\tau$. The black dotted lines indicate the point $s = 1+\tau$. }
    \label{Fig_Phi tau}
\end{figure}

\begin{remark}[Integral representation] Due to a classical integral evaluation associated with the semicircle law, the rate function \eqref{def of rate Phi_1} admits the representation 
\begin{align} \label{int rep of Phi 1}
\Phi_1(s) = \frac{s^2}{4} -\frac12 + \int_{ -2 }^2 \log \frac{1}{|s-x|} \frac{ \sqrt{4-x^2} }{2\pi} \,dx . 
\end{align}
More generally, for the elliptic Ginibre ensembles, the rate function \eqref{def of rate Phi_tau} admits an analogous two-dimensional integral representation
\begin{equation} \label{int rep of Phi tau}
\Phi_\tau(s) = \frac{s^2}{2(1+\tau)} -\frac12 +  \int_E\log \frac{1}{|s-w|}\,\frac{d^2w}{ \pi(1-\tau^2) }, 
\end{equation}
where the semicircle law is replaced by the elliptic law \eqref{def of elliptic law}.
This evaluation follows from \eqref{evluation of the log energy elliptic law} below.
In \eqref{int rep of Phi tau}, the first term on the right-hand side can be interpreted as $V(s)/2$, where $V$ is the external potential \eqref{def of V eGinibre}. We refer to Remark~\ref{Rem_obstacle} for further discussion of the potential-theoretic origin of such integral representations. 
\end{remark}

\begin{remark}[Finite-$N$ corrections] When describing the asymptotic behaviour of the spectral radius \eqref{result in CMV16 and LGMS18}, the first correction term was also derived in \cite[Eq.~(14)]{LGMS18}. This correction arises from the subexponential prefactor in Laplace's method. A similar refinement could also be obtained in our setting, since Propositions~\ref{Prop_eGinUE asymp}, \ref{Prop_eGinOE asymp} and
\ref{Prop_eGinSE asymp} provide precise asymptotic expansions of the
one-point functions that serve as the integrands in Laplace's method.  
\end{remark}

\subsection{Generalisation to arbitrary domains}

Our results in Theorem~\ref{Thm. LDP maxEV} can in fact be formulated in a more general setting. For this purpose, let $U \subset \mathbb{C}$ be a domain in the complex plane such that $U \cap E = \emptyset$, where $E$ denotes the support of the elliptic law \eqref{def of elliptic law}.
The probabilities in \eqref{def of extremal prob eGinibre} then arise as 
special cases of the more general probability that at least one eigenvalue lies in the domain $U$. 
Namely, in the spectral radius case one takes 
$U = \{ z \in \mathbb{C} : |z| \ge s \}$, whereas in the rightmost eigenvalue case one considers $U = \{ z \in \mathbb{C} : \re z \ge s \}$. 

The proof of our main result works for such a general formulation, albeit in this case, the resulting rate function becomes slightly more involved compared to that in Theorem~\ref{Thm. LDP maxEV}. To describe it, we first introduce 
\begin{equation} \label{def of Omega explicit}
\Omega(z) := \frac{1}{1-\tau^2}\Big( |z|^2-\tau \re z^2\Big) -  \re  \bigg[ \frac{2z}{ z+\sqrt{z^2-4\tau} } \bigg]- 2 \log \bigg| \frac{z + \sqrt{z^2 - 4\tau}}{2} \bigg|, 
\end{equation}
where $z \in E^c$. Here, the square root is taken with the principal branch on $\mathbb{C} \setminus [-2\sqrt{\tau}, 2\sqrt{\tau}]$; see Figure~\ref{Fig_Omega} for the graphs of $z \mapsto \Omega(z).$ 
This function satisfies the following properties: for $z \in E^{c}$, one has $\Omega(z) \ge 0$, with equality on $\partial E$; see \cite[Lemma D.1]{LR16}. Then as an extension of Theorem~\ref{Thm. LDP maxEV}, we have the following. 

\begin{figure}[t]
\begin{subfigure}{0.55\textwidth}
        \begin{center}
            \includegraphics[width=\linewidth]{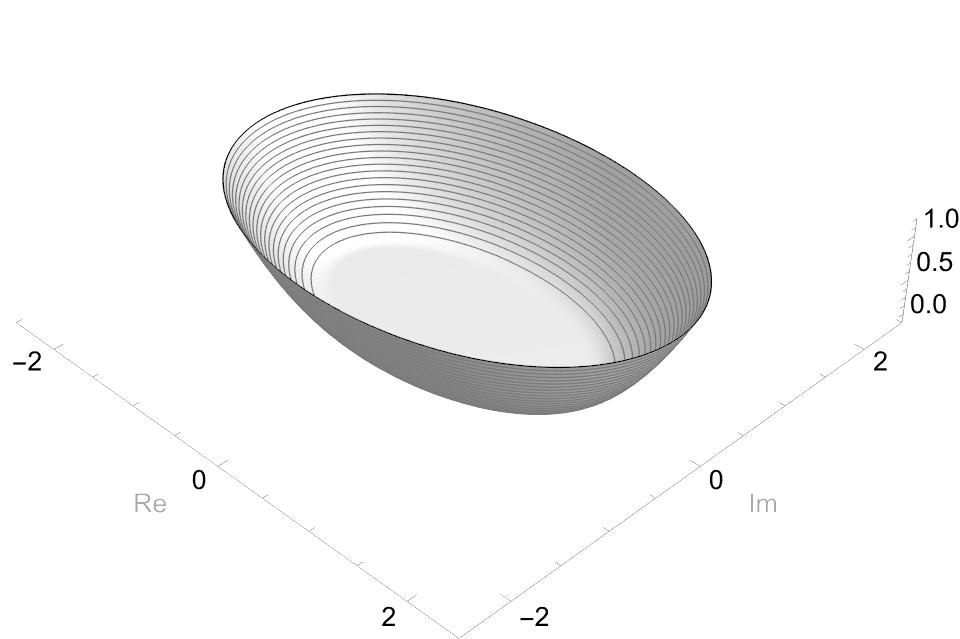}
        \end{center}
        \subcaption{ Surface plot }
    \end{subfigure}
    \qquad
    \begin{subfigure}{0.35\textwidth}
        \begin{center}
            \includegraphics[width=\linewidth]{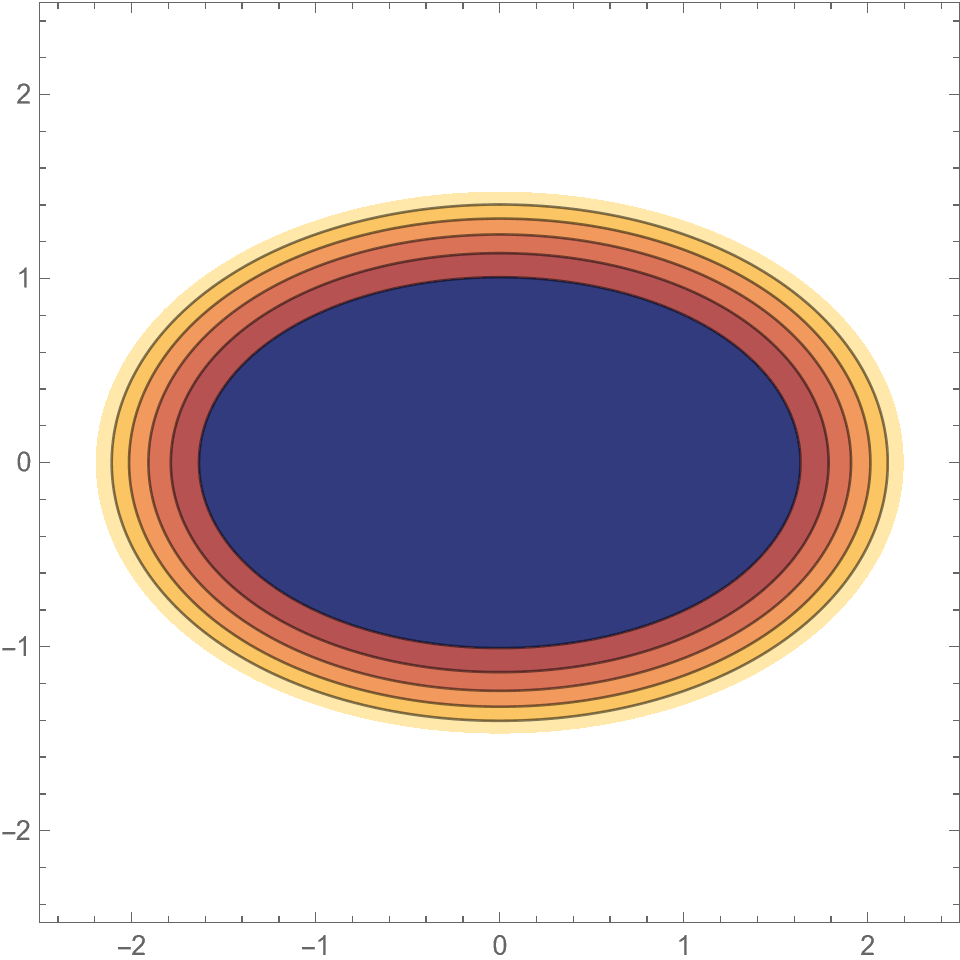}
        \end{center}
        \subcaption{ Contour plot  }
    \end{subfigure}
    \caption{ Surface and contour plot of $z \mapsto \Omega(z)$ in $E^c$, where $\tau = 0.3$.}
    \label{Fig_Omega}
\end{figure}

\begin{theorem}[\textbf{Upper tail large deviation probabilities in a general domain}] \label{Thm. LDP maxEV_general domain}
Let $\tau \in [0,1)$ be fixed, and let $N \in \mathbb{N}$. In the eGinOE case, we further assume that $N$ is even. Suppose $U \subset \C$ is measurable and bounded away from $E$, where $E$ is given by \eqref{def of elliptic law}. Then for the eGinUE and eGinSE, we have 
    \begin{equation}
    \lim_{N \to \infty} \frac{1}{\beta N}\log \P\Big[ \exists\, 1 \le j \le N \ \textup{such that } z_j \in U \Big] = - \frac12 \essinf_{ z \in U } \Omega(z), 
    \end{equation} 
where $\beta$ is given by \eqref{def of beta for eGinibres}, and for the eGinOE, we have 
\begin{equation}
    \lim_{N \to \infty} \frac{1}{N}\log \P\Big[ \exists\, 1 \le j \le N \ \textup{such that } z_j \in U \Big] = - \min \Big\{\essinf_{ z \in U \setminus \R } \Omega(z),  \frac{1}{2}\essinf_{z \in U \cap \R} \Omega(z) \Big\}.
\end{equation}
Here, $\essinf$ denotes the essential infimum with respect to the Lebesgue measure on the underlying space (either $\C$ or $\R$) with the convention $\essinf \emptyset = \infty$.
\end{theorem}

Note that Theorem~\ref{Thm. LDP maxEV_general domain} extends 
Theorem~\ref{Thm. LDP maxEV}. 
The latter is recovered by taking $U = \{ z \in \mathbb{C} : |z| \ge s \}$ or $U = \{ z \in \mathbb{C} : \re z \ge s \}$, together with the observation  
\begin{equation}  \label{relation btw Phi and Omega}
\Phi_\tau(s) = \frac12\, \Omega(s)  
\end{equation} 
and the fact that for such domains, 
\begin{equation}
\inf_{ z \in U } \Omega(z)= \Omega(s); 
\end{equation}
see Lemma~\ref{Lem_omega for special U} below.
 
\begin{remark}[Rate function as a solution to an obstacle problem] \label{Rem_obstacle}
We now discuss the potential-theoretic interpretation of the rate 
function. 
For this, we briefly recall some basic notions from logarithmic 
potential theory \cite{ST97}.
Let $Q : \mathbb{C} \to \mathbb{R}$ be a general external potential 
satisfying standard growth and regularity assumptions. 
The equilibrium measure $\sigma_Q$ is defined as the unique minimiser 
of the weighted logarithmic energy functional 
\begin{equation} \label{def of log energy}
	I_Q[\mu]:= \int_{ \mathbb{C}^2 } \log \frac{1}{ |z-w| }\, d\mu(z)\, d\mu(w) +\int_{ \mathbb{C} } Q(z) \,d\mu(z). 
\end{equation}
Under suitable regularity conditions on $Q$, the equilibrium measure 
takes the form 
\begin{equation} \label{def of eq msr general}
d \sigma_Q(z) =  \Delta Q(z)\, \mathbbm{1}_{S_Q}(z) \, \frac{d^2z}{\pi}, \qquad \Delta=\partial \bar{\partial},   
\end{equation}
where the compact set $S_Q$, known as the \emph{droplet}, is the support of 
$\sigma_Q$.
In particular, when the potential $Q$ is given by $V$ in 
\eqref{def of V eGinibre}, the associated equilibrium measure coincides 
with the elliptic law \eqref{def of elliptic law}.

The \textit{obstacle function} $\check{Q}$ associated with a potential 
$Q$ is defined as the maximal subharmonic function satisfying
\begin{equation}
\check{Q}(z) \le Q(z) \quad \text{on } \mathbb{C},
\qquad 
\check{Q}(z) \le \log |z|^2 + O(1) \quad \text{as } |z| \to \infty.
\end{equation} 
Equivalently, it admits the representation
\begin{equation} \label{def of Q check}
\check{Q}(z)=\gamma-2\,U_{Q}(z), \qquad U_{Q}(z):=\int_\C \log \frac{1}{|z-w|}\,d\sigma_{Q}(w), 
\end{equation}
where  $\gamma$ is a unique constant satisfying $2U_{Q}+Q =\gamma$ on $S_Q$ and $2U_{Q}+Q \ge \gamma$ on $\C$ quasi-everywhere. 

For the elliptic law \eqref{def of elliptic law}, by using 
\cite[Lemma 2.4]{By24a} and \cite[Lemma 3.3 with $c=0$]{BY25}, one can show that 
\begin{align} \label{evluation of the log energy elliptic law}
\int_E\log \frac{1}{|z-w|}\,\frac{d^2w}{ (1-\tau^2)\pi }   = \frac12-  \re  \bigg[ \frac{z (z-\sqrt{z^2-4\tau}) }{ 4\tau } \bigg] -   \log \bigg| \frac{z + \sqrt{z^2 - 4\tau}}{2} \bigg| .
\end{align}
Then by the representation \eqref{def of Q check}, for the potential $V$ defined in \eqref{def of V eGinibre}, we have 
\begin{equation}
\check{V}(z)=   \re  \bigg[ \frac{2z}{ z+\sqrt{z^2-4\tau} } \bigg] + 2 \log \bigg| \frac{z + \sqrt{z^2 - 4\tau}}{2} \bigg|.
\end{equation}
Combining this with \eqref{def of Omega explicit}, it follows that $\Omega$ admits the expression
\begin{equation} \label{def of Omega interms of obstacle}
\Omega(z) = V(z)-\Check{V}(z).
\end{equation}  
Therefore, one observes that the large deviation rate function admits a natural interpretation in terms of the obstacle function associated with the underlying potential.
\end{remark}

\subsection{Asymptotics of the one-point functions outside the droplet}

The key ingredient in the proof of Theorem~\ref{Thm. LDP maxEV_general domain} is the asymptotic behaviour of the one-point function of the elliptic Ginibre ensembles. 

For the eGinUE and eGinSE, we denote their one-point functions by $R_N^{\mathbb C}$ and $R_N^{\mathbb H}$, respectively. These are characterised by
\begin{equation} \label{eq. def 1-point func UE and SE}
\mathbb{E} \bigg[\sum_{j=1}^N f(z_j) \bigg] = \int_{\mathbb C} f(z) R_N^{\mathbb C}(z) \,d^2 z,
\qquad
\mathbb{E} \bigg[ \sum_{j=1}^N f(z_j) \bigg] = \int_{\mathbb C} f(z) R_N^{\mathbb H}(z)\,d^2 z,
\end{equation} 
for any suitable test function $f:\mathbb C\to\mathbb C$.
For the eGinOE, we denote by $R_N^{\mathbb R,\mathrm{r}}$ and $R_N^{\mathbb R,\mathrm{c}}$ the real and complex one-point functions, which are characterised by
\begin{equation} \label{eq. def 1-point func OE}
\mathbb{E} \bigg[ \sum_{z_j \in \mathbb R} f(z_j) \bigg] =\int_{\mathbb R} f(x) R_N^{\mathbb R,\mathrm{r}}(x)\,dx,
\qquad
\mathbb{E} \bigg[ \sum_{z_j \in \mathbb C\setminus\mathbb R} f(z_j) \bigg]= \int_{\mathbb C} f(z) R_N^{\mathbb R,\mathrm{c}}(z)\,d^2 z
\end{equation} 
for suitable test functions. 

In Propositions~\ref{Prop_eGinUE asymp}, \ref{Prop_eGinOE asymp}, and \ref{Prop_eGinSE asymp} below, we establish that as $N \to \infty$, 
the one-point functions exhibit exponential asymptotics 
outside the droplet $E$; see also Figures~\ref{Fig_numerics eGinUE}, ~\ref{Fig_numerics eGinOE} and ~\ref{Fig_numerics eGinSE}. 
More precisely, for $z \in E^{c}$,
\begin{equation} \label{leading order asymp of eGinUE eGinSE}
R_N^{ \mathbb{C} }(z) = \exp\Big( -N \Omega(z) +o(N) \Big),  \qquad R_N^{ \mathbb{H} }(z) = \exp\Big( -2N \Omega(z) +o(N) \Big)
\end{equation}
for the eGinUE and eGinSE, respectively, where $\Omega$ is given by \eqref{def of Omega explicit}. 
For the eGinOE, we have, for $x \in \mathbb R$ with $|x| > 1+\tau$, and for $z \in E^{c} \setminus \mathbb R$,
\begin{equation} \label{leading order asymp of eGinOE}
R_N^{\mathbb R,\mathrm{r}}(x) =  \exp\Big( -\frac{N}{2} \Omega(x) +o(N) \Big), \qquad R_N^{\mathbb R,\mathrm{c}}(z) =  \exp\Big( -N \Omega(z) +o(N) \Big). 
\end{equation}
These asymptotic behaviours explain why the function $\Omega$ naturally appears in the rate function in Theorem~\ref{Thm. LDP maxEV_general domain}.  

We stress here that the asymptotic behaviour of $R_N^{\mathbb R,\mathrm{r}}(x)$ was previously derived in \cite{BFK21,Fy16}; see also \cite[Remarks~3.4 and~3.10]{BL24}. 
In Propositions~\ref{Prop_eGinUE asymp}, \ref{Prop_eGinOE asymp}, and \ref{Prop_eGinSE asymp}, we derive not only the leading-order asymptotics \eqref{leading order asymp of eGinUE eGinSE} and \eqref{leading order asymp of eGinOE}, but also the precise expansion up to and including the constant-order term. 

We also emphasise that in the eGinOE case, the large deviation rate function originates from $R_N^{\mathbb R,\mathrm{r}}$, rather than from $R_N^{\mathbb R,\mathrm{c}}$. 
From a probabilistic point of view, this indicates that when the large deviation event occurs, it is typically a real eigenvalue that enters the unexpected domain $U$ (in Theorem~\ref{Thm. LDP maxEV_general domain}), rather than a complex eigenvalue; see \cite{XZ24} for a similar discussion. 
%Finally, by comparing the asymptotic behaviours of $R_N^{\mathbb C}$, $R_N^{\mathbb H}$,  and $R_N^{\mathbb R,\mathrm{r}}$, one observes that the value of $\beta$ in \eqref{def of beta for eGinibres} naturally emerges from the corresponding leading coefficients.

\begin{remark}[Generalisation to random normal matrix models]
For the symmetry class corresponding to the GinUE ($\beta=2$), 
the results presented above can be formulated in a more general framework, 
namely in the setting where the quadratic potential $|z|^2$ in 
\eqref{def of JPDF GinUE} is replaced by a general external potential $Q(z)$ 
satisfying suitable regularity and growth assumptions. 
Such models are known as random normal matrix models or determinantal Coulomb gases.

Owing to the determinantal structure and the representation in terms of planar orthogonal polynomials, one may derive that the associated one-point function $R_{N,Q}$ governs the asymptotic behaviour of the form
\begin{equation} \label{R(z) asymp general Q}
R_{N,Q}(z) = \exp\Big( -N\big(Q(z)-\check{Q}(z)\big) + o(N) \Big),
\end{equation}
for $z$ outside the droplet, where $\check{Q}$ denotes the obstacle function defined in \eqref{def of Q check}.

For a radially symmetric potential $Q(z)=Q(|z|)$, the obstacle function
can be computed explicitly using Jensen's formula; see e.g.
\cite[Theorem 6.1 in Section IV.6]{ST97} and \cite[Eq.~(2.4)]{BS23}.
In particular, when $Q(z)=|z|^2$, we have $\check{Q}(z)=2\log|z|+1$ and hence the asymptotic behaviour \eqref{R(z) asymp general Q} is consistent with the known result in the literature \cite[Proposition~1 with $\Gamma=2$]{CFTW15}. 

The upper bound of the form \eqref{R(z) asymp general Q} is indeed known under the name of the exterior estimate; see e.g. \cite[Lemma~3.3]{AKM19} and \cite[Section 4.1.1]{AHM15}. 
On the other hand, it is expected that the corresponding lower bound can be derived from recent development of the theory of planar orthogonal polynomials \cite{HW21}. 
Together, these results suggest that for the random normal matrix model, 
Theorem~\ref{Thm. LDP maxEV_general domain} is expected to hold with $\Omega$ replaced by $Q(z)-\check{Q}(z)$, which is consistent with our formulation \eqref{def of Omega interms of obstacle}.

Nevertheless, we do not pursue this direction further in the present work. Indeed, the above formulation relies crucially on the determinantal structure and therefore applies only to the case $\beta=2$. 
For the other symmetry classes ($\beta=1,4$), a comparable theory has not yet been developed. Instead, our focus here is on the elliptic Ginibre ensembles, where we are able to obtain a comprehensive understanding across all three symmetry classes. 
\end{remark}

\subsection*{Plan of the paper} The rest of this paper is organised as follows. In Section~\ref{Sec. 1-point func and summation formula}, we collect several exact identities for the one-point functions of the three symmetry classes of elliptic Ginibre ensembles. In Section~\ref{Section_large N}, we combine these identities with asymptotic properties of the relevant orthogonal polynomials to derive the large-$N$ behaviour of the one-point functions; see Propositions~\ref{Prop_eGinUE asymp}, \ref{Prop_eGinOE asymp}, and \ref{Prop_eGinSE asymp}. Finally, these results are used to establish our main theorems, Theorems~\ref{Thm. LDP maxEV} and \ref{Thm. LDP maxEV_general domain}.

\section{Exact formulas for finite-\texorpdfstring{$N$}{N}} \label{Sec. 1-point func and summation formula}

As mentioned above, the availability of exact finite-$N$ formulas plays a crucial role in enabling a precise asymptotic analysis. 
For the elliptic Ginibre ensembles, substantial progress has been made in the exact computation of correlation functions, in particular through the underlying determinantal and Pfaffian integrable structures.  

The finite-$N$ analysis we discuss here proceeds in two main steps. The first step is to obtain explicit representations of the one-point functions. 
Such representations have been developed using the determinantal and Pfaffian structures of the ensembles, together with the orthogonal and skew-orthogonal polynomial formalism; see in particular \cite{BS09,FN08} and \cite{Ka02} for the integrable structures of the Pfaffian point processes associated with the eGinOE and eGinSE, respectively. 
Within these frameworks, the one-point functions can be expressed explicitly in terms of large summations involving Hermite polynomials.
However, these summation formulas are not immediately amenable to asymptotic analysis, as their large summation structure obscures the dominant contributions. 
To overcome this difficulty, in the second step, we adopt the approach developed in \cite{LR16,BE23,FN08,BKLL23}, which avoids a direct analysis of the large summations. Instead, it exploits an appropriate operator framework to compute the one-point functions in a more tractable form.  

\subsection{Planar orthogonal and skew-orthogonal polynomials}

To introduce the integrable structures underlying the eigenvalue distributions of the elliptic Ginibre ensembles, we begin by fixing some notation. Recall that the external potential $V$ is given by \eqref{def of V eGinibre}. We then define the associated weight functions
\begin{equation} \label{def of omega eGinUE eGinSE}
    \omega_N^\C(z):=  e^{-NV(z)}, \qquad  \omega_N^\HH(z) := e^{ -2NV(z) },
\end{equation}
corresponding to the eGinUE and eGinSE, respectively.  
For the eGinOE, the weight function is slightly more involved and takes the form
\begin{equation} \label{def of omega eGinOE}
 \omega_N^\R(z) := e^{\frac{N}{1+\tau} ((\im z)^2-(\re z)^2) } \erfc\Big( \sqrt{\frac{2N}{1-\tau^2}} \ \im z  \Big),
\end{equation}
for $\im z\geq0$, and defined by the symmetry $z\mapsto\overline{z}$ for $\im z<0$, where 
$$
\erfc(z) = \frac{2}{ \sqrt{\pi} } \int_z^\infty e^{-t^2}\,dt
$$
is the complementary error function. 

In addition to the weight functions, key ingredients in the description of the one-point functions are the families of associated planar orthogonal and skew-orthogonal polynomials. For the elliptic Ginibre ensembles, these polynomials can be expressed in terms of the Hermite polynomials 
\begin{equation}
H_n(x)  := (-1)^n e^{x^2} \frac{d^n}{dx^n} e^{-x^2}. 
\end{equation}
For a positive integer $N$ and a non-negative integer $j$, it is convenient to introduce the monic polynomials 
\begin{equation} \label{def of monic Hermite phi}
 \phi_{N,j}(z):= \Big( \frac{\tau}{2N} \Big)^{j/2} H_j\Big( \sqrt{\frac{N}{2\tau}} z \Big). 
\end{equation}

With this notation in place, we first define the orthogonal polynomials and their corresponding norms associated with the eGinUE: 
\begin{equation} \label{def of OP eGinUE}
    p_{N,j}^\C(z) := \phi_{N,j}(z), \qquad h_{N,j}^\C := \frac{j!}{N^{j+1}} \pi\sqrt{1-\tau^2}.
\end{equation}
For the eGinSE and eGinOE, the situation is more involved, as the relevant families form skew-orthogonal polynomials, whose construction requires a separate analysis. 
For the eGinSE, we define
\begin{equation} \label{def of SOP eGinSE}
p_{N,2j+1}^{\tS}(z) := \phi_{2N,2j+1}(z), \qquad p_{N,2j}^{\tS}(z) := \sum_{l=0}^{j}\frac{(2j)!!}{(2l)!!}\frac{1}{(2N)^{j-l}}\phi_{2N,2l}(z). 
\end{equation}
It was shown in \cite{Ka02} (see also \cite{AEP22}) that these form a family of skew-orthogonal polynomials (we omit the precise definition here), with skew-norm 
\begin{equation} \label{def of skew norm eGinSE}
h_{N,j}^{\tS} := \frac{(2j+1)!}{(2N)^{2j+2}}2(1-\tau)\pi\sqrt{1-\tau^2}. 
\end{equation} 
Finally, for the eGinOE, the skew-orthogonal polynomials are given in \cite{FN08} by 
\begin{equation}  \label{eq. eGinOE SOP def}
p_{N,2j}^\R(z) := \phi_{N,2j}(z), \qquad p_{N,2j+1}^\R(z) := \phi_{N,2j+1}(z) - \frac{2j}{N} \phi_{N,2j-1}(z),
\end{equation}
with skew-norm 
\begin{equation} \label{def of skew norm eGinOE}
h_{N,j}^\R := \frac{(2j)!}{N^{2j+3/2}}\sqrt{2\pi} (1+\tau). 
\end{equation} 

These families of polynomials introduced above serve as the building blocks for expressing the one-point functions of the eGinUE, eGinSE, and eGinOE. 

\subsection{Exact identities of the eGinUE correlation functions}

We begin with the exact formulas for the eGinUE case. Using \eqref{def of OP eGinUE}, it is convenient to introduce  
\begin{equation} \label{def of KNM eGinUE}
    K_{N,M}^\C(z,w) := \sum_{j=0}^{M-1} \frac{p_{N,j}^\C(z) p_{N,j}^\C(w)}{h_{N,j}^\C}. 
\end{equation}
It is well known that the eigenvalues of the eGinUE form a determinantal point process with (unweighted) correlation kernel
\begin{equation} \label{def of KNN eGinUE}
K_N^\C(z,w) := K_{N,N}^\C(z,w). 
\end{equation} 
Consequently, the one-point function $R_N^{\mathbb C}(z)$ admits the representation 
\begin{equation} \label{eq. 1-point func eGinUE R-N-C}
    R_N^\C(z) = \omega_N^\C(z) K_N^\C(z,\overline{z}) , 
\end{equation}
where $\omega_N^\C$ is given by \eqref{def of omega eGinUE eGinSE}. Such a determinantal formula has been used to derive several interesting properties of the eGinUE, including two-point observables; see e.g. \cite{SKSK24}.

The main idea due to Lee and Riser in \cite[Proposition 2.3]{LR16} for the asymptotic analysis of the correlation kernel is the following: upon applying a suitable differential operator, the kernel can be rewritten in terms of only the last few orthogonal polynomials of highest degree. We recall the precise statement. 

\begin{lemma}[cf. Proposition 2.3 in \cite{LR16}] \label{lem. CD eGinUE}
    For any positive integers $N$ and $M$, we have
    \begin{equation} \label{eq. CD formula K-N-M-C}
        \frac{\partial}{\partial z} \Big[ K_{N,M}^\C(z,\overline{z}) \omega_N^\C(z) \Big] =\frac{N}{1-\tau^2}  \frac{1}{h_{N,M-1}^\C} \Big( \tau p_{N,M}^\C(z) \overline{p_{N,M-1}^\C(z)} - p_{N,M-1}^\C(z) \overline{p_{N,M}^\C(z)} \Big) \omega_N^\C(z).
    \end{equation}
    In particular, we have 
    \begin{equation} \label{CD eGinUE}
        \frac{\partial}{\partial z} R_N^\C(z) = \frac{N}{1-\tau^2} \frac{1}{h_{N,N-1}^\C} \Big(\tau p_{N,N}^\C(z)\overline{p_{N,N-1}^\C(z)} - p_{N,N-1}^\C(z)\overline{p_{N,N}^\C(z)}\Big) \omega_N^\C(z) .
    \end{equation}
\end{lemma}

This lemma provides an effective tool for analysing the large-$N$ asymptotic behaviour of the correlation kernel across various regimes (see also \cite{ADM22} for an alternative approach). Moreover, its applicability extends beyond the unitary case: as we will discuss in the following subsections, an analogous strategy can also be implemented for the other symmetry classes by exploiting suitable connection formulas.

\subsection{Exact identities of the eGinSE correlation functions} In contrast to the eGinUE whose eigenvalues form a determinantal point process, the eigenvalues of the eGinSE constitute a Pfaffian point process. The associated pre-kernel is given by 
\begin{align} \label{def of pre kernel eGinSE}
\kappa_N^\HH(z,w) := \sum_{j=0}^{N-1} \frac{p_{N,2j+1}^\HH(z)p_{N,2j}^\HH(w) - p_{N,2j}^\HH(z)p_{N,2j+1}^\HH(w)}{ h_{N,j}^\HH },
\end{align}
where $p_j^{\HH}$ and $h_{N,j}^{\HH}$ are defined in \eqref{def of SOP eGinSE} and \eqref{def of skew norm eGinSE}, respectively. 
Consequently, the one-point function $R_N^\HH(z)$ admits the representation
\begin{equation}
    R_N^\HH(z) = (\overline{z}-z) \omega_{N}^\HH(z) \kappa_N^\HH(z,\overline{z}),
\end{equation}
where the weight function $\omega_N^\HH(z)$ is given by \eqref{def of omega eGinUE eGinSE}.

Observe that, since the polynomials $ p_{N,j}^{\mathbb H} $ are expressed as finite linear combinations of Hermite polynomials, the pre-kernel \eqref{def of pre kernel eGinSE} admits a representation as a double summation. This additional layer of summation creates a further difficulty in the asymptotic analysis. 
To overcome this obstacle, the key idea developed in \cite[Proposition 1.1]{BE23} is to apply a suitable differential operator, thereby establishing a connection formula with the eGinUE kernel.

\begin{lemma}[cf. Proposition 1.1 and Lemma 3.2 in \cite{BE23}] \label{lem. CD formula for eGinSE kappa-N and E-N}
    For any positive integer $N$, we have
    \begin{equation}
        \Big(\frac{\partial}{\partial z} - \frac{2N}{1+\tau}z\Big) \kappa_N^\HH(z,w) = \frac{N}{(1-\tau^2)}\Big(K_{2N}^\C(z,w)-E_N(z,w)\Big),
    \end{equation}
    where $K_{2N}$ is given by \eqref{def of KNN eGinUE} with $N \mapsto 2N$ and
    \begin{equation}
    E_N(z,w) := \frac{1}{h_{2N,2N}^\C} p_{2N,2N}^\C(z) p_{N,2N-2}^\HH(w).
\end{equation}
   Moreover, we have 
    \begin{equation} 
        \Big(\frac{\partial}{\partial w} - \frac{2N}{1+\tau}w\Big) E_N(z,w) = -\frac{2N}{1+\tau}\frac{1}{h_{2N,2N}^{\C}} p_{2N,2N}^{\C}(z) p_{2N,2N-1}^{\C}(w).
    \end{equation}
\end{lemma}

Combined with Lemma~\ref{lem. CD eGinUE}, this reduction yields a representation that is amenable to asymptotic analysis. This strategy has been applied to derive various scaling limits of the eGinSE correlation functions; see \cite{BE23,BES23}.

\subsection{Exact identities of the eGinOE correlation functions}

Finally, we turn to the integrable structure of the eGinOE correlation functions.  
The eigenvalues of the eGinOE form a Pfaffian point process. However, since the spectrum consists of both real and complex eigenvalues, the one-point functions for the real and complex parts must be treated separately. 
In addition, the structure of the Pfaffian kernel depends on the parity of the dimension $N$. In what follows, we restrict our attention to the case where $N$ is even.

The correlation functions of classical asymmetric random matrices with real entries can be expressed as a rank-one perturbation of their complex counterparts; see \cite{FN08,BN25}. 
For the one-point function of the real eigenvalues, it was shown in \cite{FN08} that
\begin{equation}
    R_N^{\R, \rmr}(x) = \sqrt{\frac{1-\tau^2}{2\pi N}} \omega_N^\R(x) K_{N,N-1}^\C(x,x) + \frac{1}{h_{N,N-2}^\R} \sqrt{\omega_N^\R(x)} \phi_{N,N-1}(x) \int_0^x du \sqrt{\omega_N^\R(u)} \phi_{N,N-2}(u),
\end{equation}
where $\omega_N^\R$ and $K_{N,M}^\C(z,w)$ are defined in \eqref{def of omega eGinOE} and \eqref{def of KNM eGinUE}, respectively. 
Using this, various asymptotic properties of the real eigenvalues such as the limiting density and scaling limits were derived in \cite{FN08,BKLL23,BL24}.  

For the one-point function of the complex eigenvalues, it was also shown in \cite{FN08} (for even $N$) that
\begin{equation} \label{eq. 1-point func R-N-R-rmc as S-N-rmc}
    R_N^{\R, \rmc}(z) =  \omega_N^\R(z) S_N^\rmc(\overline{z},z)
\end{equation}
for $\im z>0$, and by the symmetry $z\mapsto\overline{z}$ for $\im z<0$, where
\begin{align}
    S_N^\rmc(z,w) := i \sum_{j=0}^{N/2-1} \frac{p_{N,2j+1}^\R(z) p_{N,2j}^\R(w) - p_{N,2j+1}^\R(w) p_{N,2j}^\R(z)}{ h_{N,j}^\R}.  \label{eq. S-N-c def}
\end{align}
Here, $p_{N,j}^\R$ and $h_{N,j}^\R$ are given by \eqref{eq. eGinOE SOP def} and \eqref{def of skew norm eGinOE}, respectively.
 
The following lemma is obtained in \cite{AP14} showing that $ R_N^{\mathbb R,\mathrm c} $ likewise admits a representation as a finite-rank perturbation of its complex counterpart. 

\begin{lemma}[cf. Eq.~(5.18) in \cite{AP14}] \label{lem. CD formula eGinOE S-N-c}
    For even $N$, we have
    \begin{equation}
        S_N^\rmc(z,w) = i \sqrt{\frac{\pi N}{2 (1-\tau^2)}}  \bigg( (z-w)K_{N,N-1}^\C(z, w) - \tau \frac{p_{N,N-1}^\C(z) p_{N,N-2}^\C(w) - p_{N,N-1}^\C(w) p_{N,N-2}^\C(z)}{ h_{N,N-2}^\C}  \bigg),
    \end{equation}
    where $K_{N,M}^\C$ is given in \eqref{def of KNM eGinUE}.
\end{lemma}

\section{Large-\texorpdfstring{$N$}{N} asymptotic analysis} \label{Section_large N}

This section is devoted to the asymptotic behaviour of the one-point functions of the eigenvalues of the elliptic Ginibre ensembles outside the droplet, as well as to the proofs of Theorems~\ref{Thm. LDP maxEV} and 
\ref{Thm. LDP maxEV_general domain}.

For this analysis, we make use of asymptotic expansions of the Hermite polynomials, which we now recall.
For this purpose, we define 
\begin{equation} \label{def of psi}
    \psi(z) := \frac{z + \sqrt{z^2 - 4\tau}}{2}, 
\end{equation}
where the square root is taken with the principal branch on $\mathbb{C} \setminus [-2\sqrt{\tau}, 2\sqrt{\tau}]$. This function is the inverse of the Joukowsky transform. Namely, it is a conformal map from the exterior of the unit disc onto the exterior of the ellipse $E$.
Next, we define
\begin{equation} \label{def of g}
g(z) := \frac{z}{2\psi(z)} + \log\psi(z).
\end{equation} 
Then the function $\Omega$ in \eqref{def of Omega explicit} can be expressed as
\begin{equation}
\Omega(z) = V(z)-2\re g(z),
\end{equation}
where $V$ is given by \eqref{def of V eGinibre}.  
 
Asymptotic expansions of the Hermite polynomials are well known in the 
literature. We utilise the expansions given in \cite{DKMVZ99, LR16}. 
Recall that the monic polynomial $\phi_{N,j}$ is given by \eqref{def of monic Hermite phi}. 

\begin{lemma}[cf. Lemma 2.5 in \cite{LR16}] \label{Hermite asymptotics}
    Let $r$ be a fixed integer and $z \in \C \setminus E$. Then as $N \to \infty$, we have
    \begin{equation}\label{eq. Hermite asymptotics}
        \phi_{N,N+r}(z) = \sqrt{h_{N,N+r}^{\C}} \Big(\frac{N}{2\pi^3(1-\tau^2)}\Big)^{\frac{1}{4}}  \psi(z)^r \sqrt{\psi'(z)}  \, e^{N g(z)} \Big( 1 + O\Big( \frac{1}{N} \Big) \Big),
    \end{equation}
    where the error term is uniform within a set bounded away from $E$.
\end{lemma}

\subsection{Asymptotics of the one-point functions}
 
We now present the asymptotic behaviours of the one-point functions in the following Propositions~\ref{Prop_eGinUE asymp}, ~\ref{Prop_eGinOE asymp} and ~\ref{Prop_eGinSE asymp} for each eGinUE, eGinOE and eGinSE, respectively. 
We begin with the eGinUE.  

\begin{proposition}[\textbf{Asymptotics of the one-point function for the eGinUE}] \label{Prop_eGinUE asymp}
Let $\tau \in [0,1)$ be fixed. Let $N$ be a positive integer and $z \in \C \setminus E$.
    Then as $N \to \infty$, we have
    \begin{equation}
        R_N^\C(z) =  \sqrt{ \frac{N}{2\pi^3 (1-\tau^2)} } \frac{|\psi'(z)|}{ |\psi(z)|^2 } \frac{\re\psi(z)}{ \re [ \psi(z) - \psi(z)^{-1}]  } e^{- N \Omega(z)} \Big(1+O\Big(\frac{1}{N}\Big) \Big),
    \end{equation}
    where $\psi$ and $\Omega$ are defined in \eqref{def of psi} and \eqref{def of Omega explicit}.
    The error term is uniform within a set bounded away from $E$.
\end{proposition}

See Figure~\ref{Fig_numerics eGinUE} for a numerical confirmation of Proposition~\ref{Prop_eGinUE asymp}. 

\begin{figure}[t]
    \begin{subfigure}{0.4\textwidth}
        \begin{center}
            \includegraphics[width=\linewidth]{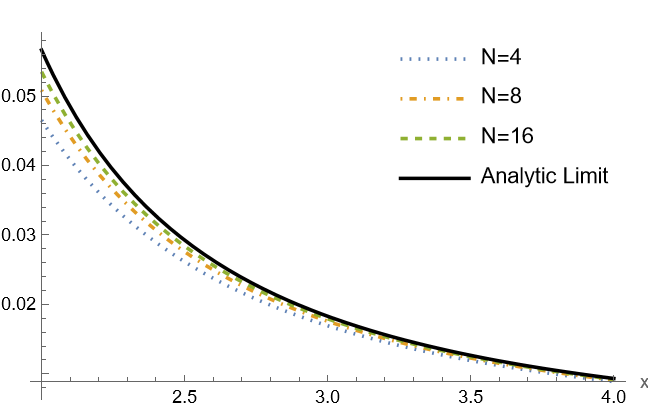}
        \end{center}
        \subcaption{  $z=x+ i \,0.35$. }
    \end{subfigure}
    \qquad
    \begin{subfigure}{0.4\textwidth}
        \begin{center}
            \includegraphics[width=\linewidth]{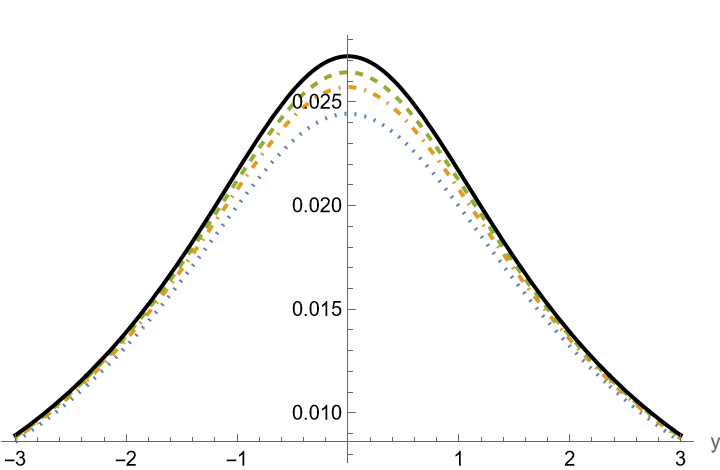}
        \end{center}
        \subcaption{ $z=2.6+iy$. }
    \end{subfigure}
     \caption{Plots of the normalised one-point function $R_N^\C(z) / (\sqrt{N} e^{-N \Omega(z)})$ for the eGinUE with $\tau=0.3$, compared with the analytic result from Proposition~\ref{Prop_eGinUE asymp}. The comparison is shown through cross-sections with either $x$ or $y$ fixed. } \label{Fig_numerics eGinUE}
\end{figure}

\begin{proof}[Proof of Proposition~\ref{Prop_eGinUE asymp}]
    We may assume that $\re z \ge 0$, since otherwise the conclusion follows by the symmetry $z \mapsto -\overline{z}$. 
Let us write $x=\re z$ and $y=\im z$. By combining Lemmas~\ref{lem. CD eGinUE} and ~\ref{Hermite asymptotics}, we have 
    \begin{align*}
        \frac{\partial}{\partial x} R_N^\C(z) &= - \frac{2N}{1+\tau} \omega_N^\C(z) \frac{1}{h_N^\C} \re\Big[ \phi_{N,N}(z) \overline{\phi_{N,N-1}(z)} \Big]
        \\
        &= - \frac{2N}{1+\tau} \sqrt{\frac{N}{2\pi^3(1-\tau^2)}} \frac{|\psi'(z)| \re\psi(z)}{|\psi(z)|^2} e^{-N \Omega(z)} \Big( 1 + O\Big( \frac{1}{N} \Big) \Big).
    \end{align*}
    Applying Laplace's method, we obtain
    \begin{align*}
        R_N^\C(z) &= - \int_x^\infty \frac{\partial}{\partial u} R_N^\C(u+iy) \,du
        \\
        &= \frac{2N}{1+\tau} \sqrt{\frac{N}{2\pi^3(1-\tau^2)}} \frac{|\psi'(z)| \re\psi(z)}{|\psi(z)|^2}  \frac{1}{N\frac{\partial}{\partial x} \Omega(z)}e^{-N \Omega(z)} \Big( 1 + O\Big( \frac{1}{N} \Big) \Big).
    \end{align*}
    We note that
    \begin{equation} \label{eq. partial-x omega}
        \frac{\partial}{\partial x} \Omega(z) = \frac{2}{1+\tau}  \re [ \psi(z) - \psi(z)^{-1}],
    \end{equation}
    which yields the desired conclusion. 
\end{proof}

Next, we present the asymptotic behaviour of the real and complex one-point functions for the eGinOE. As mentioned earlier, the asymptotic formula for $R_N^{\R,\rmr}(x)$ was already obtained in Proposition~2.1 of the supplementary material of \cite{BFK21}; see also \cite[Eqs.~(4.5), (4.6)]{Fy16}. 

\begin{proposition}[\textbf{Asymptotics of the one-point function for the eGinOE}] \label{Prop_eGinOE asymp}
 Let $\tau \in [0,1)$ be fixed and $N$ be an even positive integer.
    \begin{itemize}  
        \item \textup{(Real one-point function cf. \cite{BFK21,Fy16})} Let $x > 1+\tau$. Then as $N\to\infty$, we have
        \begin{equation}
            R_N^{\R,\rmr}(x) = \sqrt{ \frac{ N }{ 4\pi(1+\tau) } }  \frac{\sqrt {\psi'(x)}}{\psi(x)} e^{- N \Omega(x) /2} \Big( 1 + O\Big( \frac{1}{N} \Big) \Big),
        \end{equation}
        where $\psi$ and $\Omega$ are defined in \eqref{def of psi} and \eqref{def of Omega explicit}.
        The error term is uniform within a set bounded away from $[-1-\tau,1+\tau]$. 
       \smallskip 
        \item \textup{(Complex one-point function)} Let $z \in \C \setminus E$. Then as $N\to\infty$, we have
        \begin{equation} \label{eq. asymp 1-point func eGinOE cplx}
            R_N^{\R, \rmc}(z) = \frac{N}{ \pi (1-\tau^2) }  \frac{|\psi'(z)|}{ |\psi(z)|^4} \bigg| \frac{\im z \,\re \psi(z)}{\re[\psi(z) - \psi(z)^{-1}]}  - \tau \im \psi(z) \bigg| \, e^{2N \re g(z)}  \omega_N^\R(z) \Big( 1+ O\Big( \frac{1}{N} \Big) \Big),
        \end{equation}
        where $g$ is defined in \eqref{def of g}.
        The error term is uniform within a set bounded away from $E$.
    \end{itemize}
\end{proposition}

See Figure~\ref{Fig_numerics eGinOE} for a numerical verification of \eqref{eq. asymp 1-point func eGinOE cplx}. 

\begin{figure}[t]
    \begin{subfigure}{0.4\textwidth}
        \begin{center}
            \includegraphics[width=\linewidth]{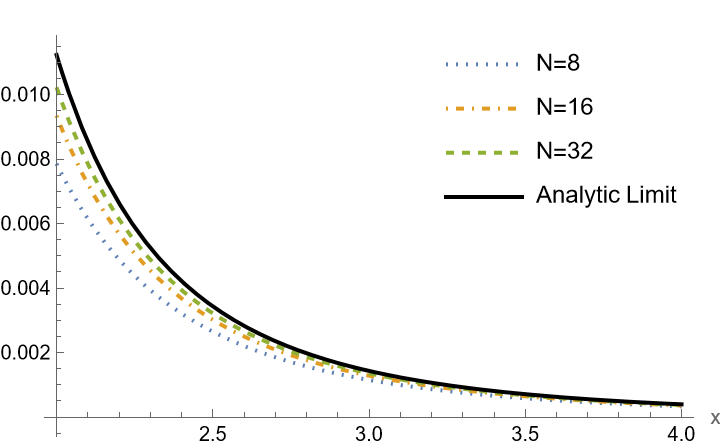}
        \end{center}
        \subcaption{ $z =x+i\,0.35$. }
    \end{subfigure}
    \qquad
    \begin{subfigure}{0.4\textwidth}
        \begin{center}
            \includegraphics[width=\linewidth]{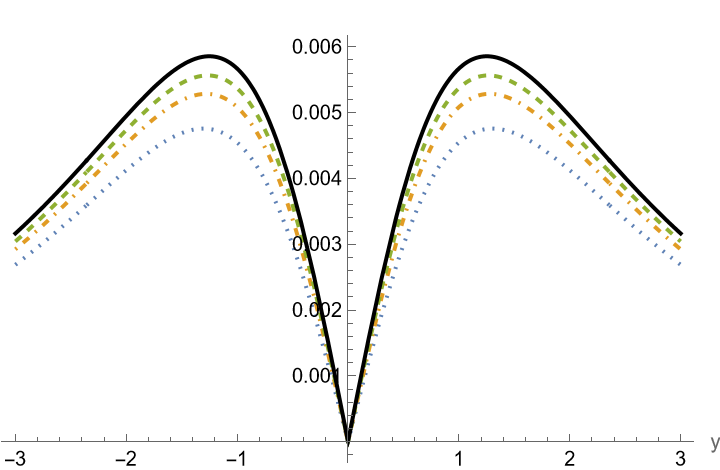}
        \end{center}
        \subcaption{ $z = 2.6+iy$. }
    \end{subfigure}
     \caption{ The same figure as in Figure~\ref{Fig_numerics eGinUE}, now for $R_N^{\R,\rmc}(z) / (N e^{2N\re g(z)}\omega_N^\R(z))$. } \label{Fig_numerics eGinOE}
\end{figure}

\begin{proof}[Proof of Proposition~\ref{Prop_eGinOE asymp}]
 By symmetries, we may assume $\re z, \im z \geq0$. By Lemma~\ref{lem. CD formula eGinOE S-N-c}, we have
    \begin{equation}
        S_N^\rmc(\overline{z},z) = \sqrt{ \frac{2 \pi N}{1-\tau^2}} \bigg( \im z \, K_{N,N-1}^\C(z,\overline{z}) - \tau \frac{\im[\phi_{N,N-1}(z) \phi_{N,N-2}(\overline{z})]}{ h_{N,N-2}^\C } \bigg).
    \end{equation}
   For the first term, $K_{N,N-1}^\C(z,\overline{z})$, we apply Lemmas~\ref{lem. CD eGinUE} and \ref{Hermite asymptotics}, which yields
    \begin{align*}
        \frac{\partial}{\partial x} \Big[ K_{N,N-1}(z,\overline{z}) \omega_N^\C(z) \Big] &= - \frac{2N}{1+\tau} \frac{1}{h_{N,N-2}^\C} \omega_N^\C(z) \re \Big[ \phi_{N,N-1}(z) \overline{\phi_{N,N-2}(z)} \Big]
        \\
        &= - \frac{N}{1+\tau} \sqrt{\frac{2N}{\pi^3(1-\tau^2)}} \frac{|\psi'(z)| \re[ \psi(z) ] }{ |\psi(z)|^4 } e^{- N \Omega(z)} \Big( 1 + O\Big( \frac{1}{N} \Big) \Big).
    \end{align*}
    Then, applying Laplace's method and writing $z = x+iy$ for $x,y\in\R$, we obtain 
    \begin{align*}
        y \,K_{N,N-1}(z,\overline{z})  &= - \frac{y}{ \omega_N^\C(z) }\int_x^\infty \frac{\partial}{\partial u} \Big[ K_{N,N-1}(u+i y,u+iy) \omega_N^\C(u+iy) \Big] \,du
        \\
        &= \frac{y}{1+\tau} \sqrt{\frac{2N}{\pi^3 (1-\tau^2)}} \frac{|\psi'(z)| \re \psi(z)}{ |\psi(z)|^4 \frac{\partial}{\partial x} \Omega(z)} e^{2N \re g(z)} \Big( 1 + O\Big( \frac{1}{N} \Big) \Big).
    \end{align*}
    For the second term, involving the product of Hermite polynomials, its asymptotics follow from straightforward computations using Lemma~\ref{Hermite asymptotics}, and we obtain
    \begin{equation*}
        \tau\frac{\im[\phi_{N,N-1}(z) \phi_{N,N-2}(\overline{z})]}{ h_{N,N-2}^\C }  = \tau\sqrt{\frac{N}{2\pi^3 (1-\tau^2)}}  \frac{|\psi'(z)| \im \psi(z)}{|\psi(z)|^4} e^{2N \re g(z)} \Big( 1 + O\Big(\frac{1}{N} \Big) \Big).
    \end{equation*}
    Combining the above asymptotics, it follows that 
    \begin{align*}
        S_N^\rmc(\overline{z},z) = \frac{N}{ \pi (1-\tau^2) }  \frac{|\psi'(z)|}{ |\psi(z)|^4} \bigg( \frac{2y}{1+\tau} \frac{\re \psi(z)}{\frac{\partial}{\partial x}\Omega(z)}  - \tau \im \psi(z) \bigg) e^{2N \re g(z)}.
    \end{align*}
    Substituting this into \eqref{eq. 1-point func R-N-R-rmc as S-N-rmc} together with \eqref{eq. partial-x omega}, we obtain \eqref{eq. asymp 1-point func eGinOE cplx}.
\end{proof}

Finally, we turn to the asymptotic behaviour of the one-point function for the eGinSE.

\begin{proposition}[\textbf{Asymptotics of the one-point function for the eGinSE}] \label{Prop_eGinSE asymp}
   Let $\tau \in [0,1)$ be fixed. Let $N$ be a positive integer and $z\in\C \setminus E$.
    Then as $N \to \infty$, we have
    \begin{equation}
        R_N^\HH(z) =  \frac{\im z }{1-\tau} \sqrt{ \frac{N}{\pi^3 (1-\tau^2)} } \frac{|\psi'(z)| }{ |\psi(z)^2-1|^2 } \frac{\im[\psi(z)^2]}{ \re [ \psi(z) - \psi(z)^{-1}] } e^{- 2N \Omega(z)} \Big(1+O\Big(\frac{1}{N}\Big) \Big),
    \end{equation}
    where $\psi$ and $\Omega$ are defined in \eqref{def of psi} and \eqref{def of Omega explicit}.
    The error term is uniform within a set bounded away from $E$.
\end{proposition}

See Figure~\ref{Fig_numerics eGinSE} for the numerics. 

\begin{figure}[t]
    \begin{subfigure}{0.4\textwidth}
        \begin{center}
            \includegraphics[width=\linewidth]{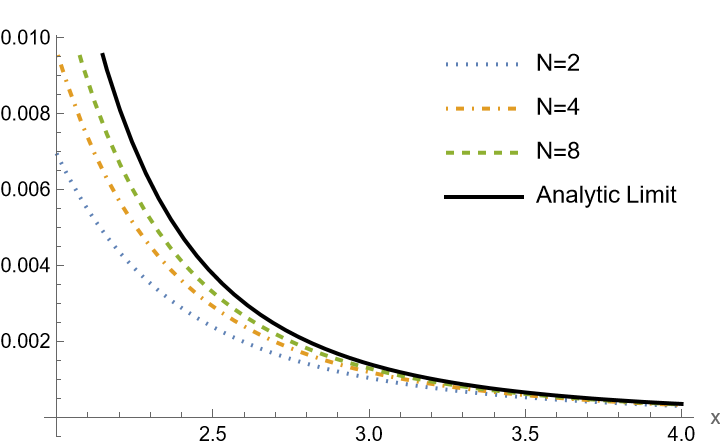}
        \end{center}
        \subcaption{ $z=x+i\,0.35$. }
    \end{subfigure}
    \qquad
    \begin{subfigure}{0.4\textwidth}
        \begin{center}
            \includegraphics[width=\linewidth]{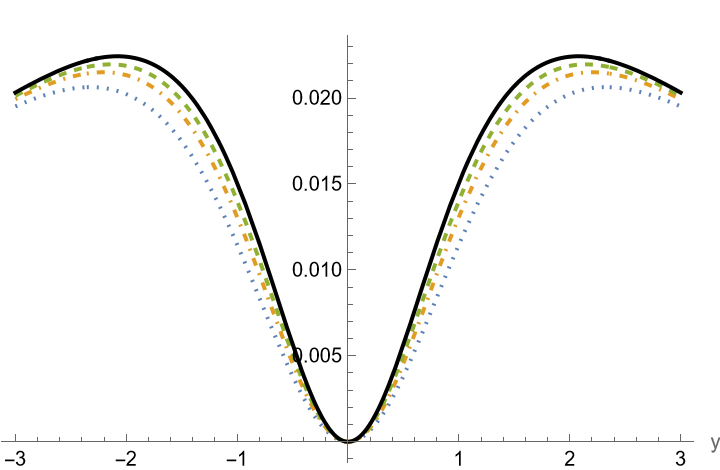}
        \end{center}
        \subcaption{ $z=2.6+iy$. }
    \end{subfigure}
     \caption{ The same figure as in Figure~\ref{Fig_numerics eGinUE}, now for $R_N^\HH(z) / (\sqrt{N}e^{-2N \Omega(z)})$. } \label{Fig_numerics eGinSE}
\end{figure}

\begin{proof}[Proof of Proposition~\ref{Prop_eGinSE asymp}]
    Let $x = \re z$ and $y = \im z$. By symmetry, we may assume $x \geq 0$.
Using Lemma~\ref{lem. CD formula for eGinSE kappa-N and E-N}, we have 
    \begin{align*}
        \frac{\partial}{\partial x} R_N^\HH(z) &=\frac{N}{1 - \tau^2} (\overline{z} - z)\omega_N^\HH(z) \Big(\overline{E_N(z, \overline{z})}-E_N(z, \overline{z})\Big).
    \end{align*}
    This reduces the problem to computing the asymptotics of $E_N(z,\overline{z})$.
Again by Lemma~\ref{lem. CD formula for eGinSE kappa-N and E-N}, we have 
    \begin{equation*}
        \frac{\partial}{\partial w}\Big[ E_N(z,w) e^{- \frac{N}{1+\tau} w^2} \Big] = - \frac{2N}{1+\tau} \frac{1}{h_{2N,2N}^\C} p_{2N,2N}^\C(z) p_{2N,2N-1}^\C(w) e^{- \frac{N}{1+\tau}w^2}.
    \end{equation*}
    We note that $E_N(z,w) e^{- \frac{N}{1+\tau} w^2}$ is an entire function that converges to $0$ as $\re w \to +\infty$.
Hence, by the method of steepest descent, we obtain 
    \begin{align*}
        E_N(z,w) e^{- \frac{N}{1+\tau} w^2} &= - \int_w^\infty \frac{\partial}{\partial \eta}\Big[E_N(z, \eta) e^{- \frac{N}{1+\tau}\eta^2}\Big] \, d\eta
        \\
        &= \frac{2N}{1+\tau} \sqrt{\frac{N}{\pi^3(1-\tau^2)}} \sqrt{\psi'(z)} e^{2N g(z)} \int_w^\infty \frac{\sqrt{\psi'(\eta)}}{\psi(\eta)} e^{2N(g(\eta) - \frac{1}{1+\tau}\eta^2)} d\eta \Big( 1+ O\Big( \frac{1}{N} \Big)\Big)
        \\
        &= \sqrt{\frac{N}{\pi^3(1-\tau^2)}} \sqrt{\psi'(z)} \frac{\sqrt{\psi'(w)} }{\psi(w)^2 -1} e^{2N (g(z) + g(w)) - \frac{N}{1+\tau} w^2} \Big( 1+ O\Big( \frac{1}{N} \Big)\Big).
    \end{align*}
    Here, the path of integration is chosen so that $\re \eta \to \infty$ while $\im \eta$ remains bounded.
Rearranging the terms yields 
    \begin{equation*}
        E_N(z, w) = \sqrt{ \frac{N}{\pi^3 (1-\tau^2)} } \sqrt{\psi'(z)} \frac{\sqrt{\psi'(w)}}{ \psi(w)^2-1} e^{2N (g(z)+ g(w))} \Big( 1+ O\Big( \frac{1}{N} \Big)\Big).
    \end{equation*}
    Substituting this into the above identity and applying Laplace's method, we obtain 
    \begin{align*}
        R_N^\HH(z) &= - \int_x^\infty \frac{\partial}{\partial u} R_N^\HH(u+ iy) \,du
        \\
        &= -\int_x^\infty \frac{4 N y}{1-\tau^2} \sqrt{\frac{N}{\pi^3 (1-\tau^2)}} \frac{|\psi'(u+iy)| \im[\psi(u+iy)^2]}{|\psi(u+iy)^2-1|^2} e^{-2N \Omega(u+iy)} du \Big( 1+ O\Big( \frac{1}{N} \Big)\Big)
        \\
        &= \frac{4N \im z }{1-\tau^2} \sqrt{\frac{N}{\pi^3(1-\tau^2)}} \frac{|\psi'(z)| \im[\psi(z)^2]}{ |\psi(z)^2-1|^2} \frac{1}{2N\frac{\partial}{\partial x}\Omega(z)} e^{-2N \Omega(z)} \Big( 1+ O\Big( \frac{1}{N} \Big)\Big), 
    \end{align*}
    which completes the proof. 
\end{proof}

\subsection{Proof of Theorems~\ref{Thm. LDP maxEV} and ~\ref{Thm. LDP maxEV_general domain}}

Using the asymptotic behaviour established in the previous section, we now prove Theorems~\ref{Thm. LDP maxEV} and ~\ref{Thm. LDP maxEV_general domain}. 
We first show the following.  

\begin{lemma} \label{Lem_omega for special U}
For $s > 1+\tau$, let $U = \{ z \in \mathbb{C} : |z| \ge s \}$ or $U = \{ z \in \mathbb{C} : \re z \ge s \}$. Then we have 
\begin{equation}
\inf_{ z \in U } \Omega(z)= \Omega(s). 
\end{equation}
\end{lemma} 
\begin{proof}
Let $z=x+iy=re^{i\theta}$. Also denote
\begin{equation*}
    u:= \psi(z) = \frac{z+\sqrt{z^2-4\tau}}{2}, \qquad z=u+\frac{\tau}{u}.
\end{equation*}
Then, we have
\begin{equation} \label{xy in u}
    x=\re\Big(u+\frac{\tau}{u}\Big)=\frac{(|u|^2+\tau)\re u}{|u|^2}, \qquad
    y=\im\Big(u+\frac{\tau}{u}\Big)=\frac{(|u|^2-\tau)\im u}{|u|^2}
\end{equation}
and 
\begin{equation*}
    \Omega(z) = \frac{1}{1-\tau^2}\Big( |z|^2-\tau \re z^2\Big) -\re\Big(\frac{z}{u}+2\log u\Big).
\end{equation*}

First fix $y\in\R$ and suppose $x>x_0>1+\tau$. Differentiating $\Omega$ with respect to $x$ yields 
\begin{equation*}
    \frac{\partial}{\partial x}\Omega(z) = \frac{2x}{1+\tau}-2\re\frac{1}{u} = \frac{2x}{1+\tau}\frac{|u|^2-1}{|u|^2+\tau}, 
\end{equation*}
where we have used \eqref{xy in u}. 
Since $\psi:\C\setminus E\to\C\setminus\D$ is conformal and $x>1+\tau$, we have $|u|>1$. Hence 
\begin{equation} \label{Omega x axis}
    \inf_{x>x_0}\Omega(x+iy)=\Omega(x_0+iy), \qquad y\in\R.
\end{equation}
By symmetry, it also holds that $$\inf_{x<-x_0}\Omega(x+iy)=\Omega(-x_0+iy).$$ 

Next fix $x>1+\tau$ and differentiate $\Omega$ with respect to $y$:
\begin{equation*}
    \frac{\partial}{\partial y}\Omega(z) = \frac{2y}{1-\tau}+2\im\frac{1}{u} = \frac{2y}{1-\tau}\frac{|u|^2-1}{|u|^2-\tau}, 
\end{equation*}
where we have used \eqref{xy in u}. 
Since $|u|>1$, it follows that 
\begin{equation} \label{Omega y axis}
    \inf_{y\in\R}\Omega(x+iy) = \Omega(x), \qquad x>1+\tau.
\end{equation} 

On the other hand, for fixed $r>1+\tau$, we compute 
\begin{equation*}
    \frac{\partial}{\partial \theta}\Omega(z) = \Big(-y\frac{\partial}{\partial x}+x\frac{\partial}{\partial y}\Big)\Omega(z) = \frac{2\tau r^2\sin2\theta}{1-\tau^2}\frac{|u|^4-1}{|u|^4-\tau^2}.
\end{equation*}
Thus, we have
\begin{equation} \label{Omega angular}
    \inf_{\theta\in[0,2\pi)}\Omega(re^{i\theta})=\Omega(r)=\Omega(-r), \qquad r>1+\tau.
\end{equation}

For $U=\{z\in\C:\re z \ge s\}$, it follows from \eqref{Omega x axis} and \eqref{Omega y axis} that $\inf_{z\in U}\Omega(z)=\Omega(s)$.
For $U=\{z\in\C:|z| \ge s\}$, using \eqref{Omega x axis} and \eqref{Omega y axis}, together with symmetry to cover the left half-plane $\{x+iy:x<0\}$, we obtain
\[
\inf_{z\in U}\Omega(z)=\inf_{|z|=s}\Omega(z).
\]
Then \eqref{Omega angular} yields the desired result, completing the proof. 
\end{proof}

Deriving the leading-order asymptotics of the upper tail from the asymptotics of the one-point function is standard in random matrix theory; see e.g. \cite[Section~14.6.4]{Fo10}. 
Our argument follows a slight modification of the approach in \cite{XZ24}. 

\begin{proof}[Proof of Theorems~\ref{Thm. LDP maxEV} and ~\ref{Thm. LDP maxEV_general domain}] 
   Note that Theorem~\ref{Thm. LDP maxEV} is an immediate consequence of Theorem~\ref{Thm. LDP maxEV_general domain}, Lemma~\ref{Lem_omega for special U}, and \eqref{relation btw Phi and Omega}. 
Therefore, it suffices to prove Theorem~\ref{Thm. LDP maxEV_general domain}. 
    
    Let $\{z_j\}_{j=1}^N$ denote the eigenvalues of the elliptic Ginibre ensemble.
    Given a measurable set $U \subset \C$, we have the bounds
    \begin{align} \label{eq. ingredient 1 Thm LDP general domain}
        \frac{1}{N}\sum_{j=1}^N \P\Big[ z_j \in U \Big] \leq \P\Big[ \exists\, 1 \le j \le N \ \textup{such that } z_j \in U \Big]
        \leq
        \sum_{j=1}^N \P\Big[ z_j \in U \Big]. 
    \end{align}
   By the definitions \eqref{eq. def 1-point func UE and SE} and \eqref{eq. def 1-point func OE}, the upper and lower bounds can be rewritten as 
    \begin{align} \label{eq. ingredient 2 Thm LDP general domain}
        \sum_{j=1}^N\P\Big[ z_j \in U \Big] = \E \Big[ \sum_{j=1}^N \mathbbm{1}_U(z_j) \Big] = \begin{dcases}
            \int_{U\cap\R} R_N^{\R,\rmr}(x) \, dx + \int_{U \setminus \R} R_N^{\R,\rmc}(z) \,d^2z, & \text{for the eGinOE},
            \\
            \int_U R_N^{\C/\HH}(z) \,d^2z, & \text{for the eGinU/SE}.
        \end{dcases}
    \end{align}
    
    Now suppose that $U$ is bounded away from $E$.
Then, as $N\to\infty$, we claim that 
    \begin{align*}
        \int_{U\cap \R} R_N^{\R, \rmr}(x) \,dx = \exp\bigg[ - \frac{N}{2} \essinf_{x \in U\cap\R} \Omega(x) + o(N) \bigg],
        \qquad
        \int_{U\setminus \R} R_N^{\R, \rmc}(z) \,d^2z = \exp\bigg[ - N \essinf_{z \in U\setminus\R} \Omega(z) + o(N) \bigg]
    \end{align*}
    for the eGinOE, and
    \begin{align*}
        \int_{U} R_N^{\C/\HH}(z) \,d^2z = \exp\bigg[ - \frac{\beta N}{2} \essinf_{z \in U} \Omega(z) + o(N) \bigg]
    \end{align*}
    for the eGinU/SE, where $\beta$ is defined in \eqref{def of beta for eGinibres}.
   Then Theorem~\ref{Thm. LDP maxEV_general domain} follows directly from this claim together with \eqref{eq. ingredient 1 Thm LDP general domain} and \eqref{eq. ingredient 2 Thm LDP general domain}.

  It remains to prove the claim. We give the details only for the eGinUE. 
The cases of the eGinSE and eGinOE follow by the same arguments. 
Let $\mu$ denote the Lebesgue measure on $\C$.
Since the claim is trivial when $\mu(U)=0$, we assume that $\mu(U)\neq 0$. 
 Let $a := \essinf_{z \in U} \Omega(z)$ and, for $\epsilon>0$, set
\[
U_\epsilon := U \cap \Omega^{-1}[a,a+\epsilon).
\]
Choose $\epsilon_1>0$ arbitrarily and let $\delta := \mu(U_{\epsilon_1})>0$.
Next choose a decreasing sequence $\epsilon_N>0$ such that $\epsilon_N\to0$ and $\mu(U_{\epsilon_N})>\delta/N$.

   By construction, $U_{\epsilon_1}$ is bounded away from both $E$ and infinity.
Since $\psi$ is the Joukowsky transform from $\C\setminus E$ to $\C\setminus\D$, Proposition~\ref{Prop_eGinUE asymp} implies that there exists a constant $c>0$ such that 
    \begin{equation*}
        R_N^\C(z) \geq c \sqrt{N} e^{-N \Omega(z)}
    \end{equation*}
  for all $z\in U_{\epsilon_1}$. Therefore we obtain 
    \begin{align*}
        \int_U R_N^\C(z) \,d^2 z &\geq \int_{U_{\epsilon_N}} R_N^\C(z) \,d^2z
        \geq   \mu(U_{\epsilon_N} ) c\sqrt{N} e^{-N(a+\epsilon_N)} d^2z >  \frac{\delta c}{\sqrt{N}} e^{-N(a+\epsilon_N)},
    \end{align*}
    which gives rise to the lower bound
    \begin{align*}
        \int_{U} R_N^\C(z) \,d^2x \geq \exp\bigg[ - \frac{\beta N}{2} \essinf_{z \in U} \Omega(z) + o(N) \bigg].
    \end{align*}

    To prove the upper bound, choose $r\in(1+\tau,\infty)$ sufficiently large so that $\Omega(r)>a$ and $\mu(U\cap B_r)>0$, where $B_r$ denotes the ball of radius $r$ in $\C$ centred at the origin. 
Again by Proposition~\ref{Prop_eGinUE asymp}, there exists a constant $c'>0$ such that
    \begin{equation*}
        R_N^\C(z) \leq c' \sqrt{N} e^{-N \Omega(z)}
    \end{equation*}
    for all $z \in U \cap B_r$.
    Then, we have
    \begin{align*}
        \int_U R_N^\C(z) \,d^2z &\leq c' \int_U e^{-N \Omega(z)} \, d^2z 
         \leq c' \bigg[ \int_{U \cap B_r} + \int_{\C \setminus B_r} \bigg] e^{-N \Omega(z)} \, d^2z
        \\
        &\leq c' \bigg[ \pi r^2 e^{-N a} + 2\pi \int_r^\infty e^{-N \Omega(x)} \, dx \bigg]
        \leq c'' e^{-Na}
    \end{align*}
    for some constant $c''>0$.
    In the third inequality, we used that $\Omega(|z|) \le \Omega(z)$ for all $z \in \C \setminus B_r$, as shown in the proof of Lemma~\ref{Lem_omega for special U}. 
    
    Combining the above, as $N\to\infty$ we obtain
    \begin{align*}
        \int_{U} R_N^\C(z) \,d^2x \leq \exp\bigg[ - \frac{\beta N}{2} \essinf_{z \in U} \Omega(z) + o(N) \bigg]. 
    \end{align*}
    This proves the claim for the eGinUE and hence completes the proof of Theorem~\ref{Thm. LDP maxEV_general domain}. 
\end{proof}

\end{document}